\documentclass[12pt]{amsart}
\usepackage{amssymb,latexsym,amsmath,enumitem,mathrsfs,geometry}
\usepackage{thmtools}
\usepackage{fullpage}
\usepackage{fancyhdr}
\usepackage{xcolor}
\usepackage{bbm}
\usepackage{stmaryrd}
\usepackage{mathrsfs}
\usepackage{hyperref}
\usepackage{lineno}
\usepackage{diagbox}
\usepackage{array}
\usepackage{tikz}
\usetikzlibrary{arrows}
\usetikzlibrary{positioning, arrows.meta, shapes.geometric}
\usepackage{float}
\usepackage{comment}
\usepackage{mathtools}
\usepackage{amsthm}
\usepackage{cleveref}

\newtheorem{theorem}{Theorem}[section]
\newtheorem*{theorem*}{Theorem}

\newtheorem{lemma}[theorem]{Lemma}
\newtheorem*{remark*}{Remark}

\newtheorem{corollary}[theorem]{Corollary}
\newtheorem{proposition}[theorem]{Proposition}

\newtheorem{remark}[theorem]{Remark}
\newtheorem{example}[theorem]{Example}

\newcommand{\eps}{\varepsilon}
\newcommand{\R}{\mathbb{R}}
\newcommand{\C}{\mathbb{C}}

\newcommand{\N}{\mathbb{N}}
\newcommand{\Z}{\mathbb{Z}}
\renewcommand{\Im}{\mathrm{Im}}
\renewcommand{\Re}{\mathrm{Re}}
\newcommand{\sinc}{\operatorname{sinc}}
\newcommand{\Poisson}{\mathcal{P}}

\begin{document}

\numberwithin{equation}{section}

\title{Local Bernstein theory, and lower bounds for Lebesgue constants}

\author{Terence Tao}
\address{UCLA Department of Mathematics, Los Angeles, CA 90095-1555.}
\email{tao@math.ucla.edu}

\thanks{}

\date{\today}

\begin{abstract} Classical (or ``global'') Bernstein theory establishes sharp control on entire functions of exponential type that are bounded and real-valued on the real axis.  We localize some of this theory to rectangular regions $\{ x+iy: x \in I, 0 \leq y \leq y_0 \}$, showing that Bernstein-type bounds with acceptable errors can continue to hold for functions holomorphic in such rectangles, bounded and real-valued on the lower edge of the rectangle, at most exponentially large on the upper edge, and at most double exponentially large on the vertical sides.

As a consequence of these bounds, we are able to localize the Erd\H{o}s lower bound  $\sup_{x \in [-1,1]} \lambda(x) \geq \frac{2}{\pi} \log n - O(1)$ on the Lebesgue constant of interpolation on $C([-1,1])$ to shorter intervals $I$ than $[-1,1]$, answering a question of Erd\H{o}s and Tur\'an.  By using suitably weighted versions of the residue theorem, we also obtain the asymptotically sharp lower bound $\int_I \lambda(x)\ dx \geq \frac{4|I|}{\pi^2} \log n - o(\log n)$ for integral variants of such constants, answering a further question of Erd\H{o}s.
\end{abstract}

\maketitle

\section{Introduction}

\subsection{Global Bernstein theory}

For any $\lambda>0$, define the \emph{Bernstein space} ${\mathcal B}^\infty_\lambda$ to be the space of entire functions $f$ which is of exponential type $\lambda$ in the sense that
\begin{equation}\label{cbound}
    |f(z)| \leq C e^{\lambda |z|}
\end{equation}
for some $C>0$ and all\footnote{Throughout this paper, $z$ (or $z'$, $z_1$, etc.) is understood to denote a complex number, while $x, y, s, t$ (or $x'$, etc.) is understood to denote a real number.} $z \in \C$ (so in particular they are entire of order $1$), and obeys the $L^\infty$ hypothesis\footnote{In the literature one also studies Bernstein spaces with other $L^p$ hypotheses on the real line, such as $L^2$ (in which case the spaces are also known as Paley--Wiener spaces), but we will not do so here.}
\begin{equation}\label{rbound}
    |f(x)| \leq A
\end{equation}
for some $A>0$ and all $x \in \R$; this is a normed vector space, with $\|f\|_{{\mathcal B}^\infty_\lambda}$ equal to the least $A$ for which \eqref{rbound} holds (i.e., the $L^\infty(\R)$ norm of $f$).  We define the \emph{real-valued Bernstein space} ${\mathcal B}^{\infty,\R}_\lambda$ to be the (real) subspace of ${\mathcal B}^\infty_\lambda$ consisting of those functions $f$ which are real-valued on the real axis, or equivalently (by the Schwartz reflection principle) obey the symmetry
\begin{equation}\label{fsymm}
f(\overline{z}) = \overline{f(z)}
\end{equation}
for all $z \in \C$, with the inherited norm $\|f\|_{{\mathcal B}^{\infty,\R}_\lambda} = \|f\|_{{\mathcal B}^\infty_\lambda}$. The spaces ${\mathcal B}^{\infty,\R}_\lambda$ have the structure of a filtered algebra, in the sense that the spaces ${\mathcal B}^{\infty,\R}_\lambda$ are increasing in $\lambda$, and the product of an element of ${\mathcal B}^{\infty,\R}_\lambda$ and an element of ${\mathcal B}^{\infty,\R}_{\lambda'}$ is an element of ${\mathcal B}^{\infty,\R}_{\lambda+\lambda'}$.

\begin{remark}
In the spirit of the Paley--Wiener theorem, one can think of ${\mathcal B}^\infty_\lambda$ as morally consisting of those elements of $L^\infty(\R)$ whose distributional Fourier transform $\hat f(\xi) \coloneqq \int_\R f(x) e^{-2\pi i \xi x}$ is supported on the interval $[-\frac{\lambda}{2\pi}, \frac{\lambda}{2\pi}]$ and obeys some technical regularity condition at the endpoints of this interval; the requirement of being real on $\R$ then translates to the symmetry $\hat f(-\xi) = \overline{\hat f(\xi)}$.  However, we will not explicitly adopt this Fourier-analytic perspective in this paper.
\end{remark}

We recall some key elements of the class ${\mathcal B}^{\infty,\R}_\lambda$:

\begin{example}[Sinusoids and sinc functions]\label{sinu-ex}  For a \emph{frequency} $\lambda > 0$, \emph{amplitude} $A > 0$ and \emph{phase} (or \emph{central point}) $x_0 \in \R$, the \emph{sinusoid}
\begin{equation}\label{acos}
    A \cos(\lambda (z - x_0))
\end{equation}
and \emph{rescaled sinc function}
\begin{equation}\label{asinc}
    A \sinc(\lambda (z - x_0))
\end{equation}
are elements of ${\mathcal B}^{\infty,\R}_\lambda$, where $\sinc(z) \coloneqq \sin(z)/z$ with the singularity at the origin removed.  As we shall see, the sinusoids are ``extremal'' elements of ${\mathcal B}^{\infty,\R}_\lambda$, while the rescaled sinc functions are useful multipliers that can partially localize elements of ${\mathcal B}^{\infty,\R}_\lambda$ to the region $x = x_0 + O(1/\lambda)$ without significantly increasing the exponential type of the function.
\end{example}

\begin{example}[Trigonometric polynomial]\label{poly-ex} For any natural number $n$ and real coefficients $a_0,\dots,a_n$ and $b_1,\dots,b_n$, the \emph{trigonometric polynomials}
\begin{equation}\label{trig}
 P(z) = a_0 + \sum_{j=1}^n a_j \cos(jz) + b_j \sin(jz)
\end{equation}
of degree (at most) $n$ are precisely the elements of ${\mathcal B}^{\infty,\R}_n$ that are $2\pi$-periodic. Thus, one can think of elements of ${\mathcal B}^{\infty,\R}_\lambda$ as ``generalized trigonometric polynomials'' of ``degree'' at most $\lambda$, which are not required to have any periodicity properties. Conversely, one can view the space of trigonometric polynomials of a given degree as a ``toy model'' for the Bernstein class ${\mathcal B}^{\infty,\R}_\lambda$.  We remark that the leading coefficients $a_n,b_n$ of such a polynomial determine the asymptotic behavior away from the real axis; for instance one has the asymptotic approximation
\begin{equation}\label{P-approx}
 P(x+iy) \approx \frac{a_n-ib_n}{2} e^{ny} e^{-inx}
\end{equation}
when $y$ is large and positive.
\end{example}

We shall informally refer to the body of results controlling the behavior of a function $f$ (or its derivative, zeroes, etc.) in a Bernstein space ${\mathcal B}^\infty_\lambda$ or ${\mathcal B}^{\infty,\R}_\lambda$ as \emph{global Bernstein theory} (to make a distinction between the local Bernstein theory we shall introduce later).  As a first example of a result in this global theory, one can interpolate the bounds \eqref{cbound}, \eqref{rbound} using a standard application of the Phragm\'en--Lindel\"of principle to obtain the unified bound
\begin{equation}\label{pl}
        |f(z)| \leq A e^{\lambda |\Im z|}.
\end{equation}
But with the additional requirement of being real-valued on $\R$, significantly more precise bounds (often with sharp constants) are available:

\begin{theorem}[Global Bernstein theory]\label{ds}  Let $\lambda, A > 0$, and let
$f \in {\mathcal B}^{\infty,\R}_\lambda$ with $\|f\|_{{\mathcal B}^{\infty,\R}_\lambda} \leq A$.
\begin{itemize}
    \item[(i)] (Bernstein inequality \cite{bern2}) For all real $x$, one has\footnote{Bernstein's result holds even if $f$ is not real-valued on $\R$.} $|f'(x)| \leq A \lambda$.
    \item[(ii)] (Boas \cite{boas}) In fact, for all real $x$, one has $\left|f(x) + \frac{i f'(x)}{\lambda}\right| \leq A$, or equivalently $f'(x)^2 + \lambda^2 f(x)^2 \leq \lambda^2 A^2$.
    \item[(iii)]  (Duffin--Schaeffer \cite{duffin})  For any complex number $x+iy$, one has $|f(x+iy)| \leq A \cosh(\lambda y)$.  If $y \neq 0$, then equality can only occur if $f$ is a sinusoid.
    \item[(iv)]  (Linear growth of zeroes) If $x_0$ is real with $f(x_0) \neq 0$, then for any $r>0$, the number of zeroes of $f$ in the disk
    $$D(x_0,r) \coloneqq \{ z: |z-x_0| < r \}$$
     is $O(1+\lambda r + \log \frac{A}{|f(x_0)|})$.
    \item[(v)]  (H\"ormander \cite{hormander}) If $x_0$ is real with $f(x_0) = \pm A$, then for all $x$ in the interval $(x_0 - \frac{\pi}{\lambda}, x_0 + \frac{\pi}{\lambda})$, one has $\pm f(x) \geq A \cos(\lambda (x-x_0))$.  If $x \neq x_0$, then equality can only occur if $f$ is a sinusoid.
    \item[(vi)]  (Zero spacing near maximum) If $x_0$ is real and $|f(x_0)|=A$, then there are no zeroes of $f$ in the interval $[x_0 - \frac{\pi}{2\lambda}, x_0 + \frac{\pi}{2\lambda}]$ except possibly at the endpoints, and the latter can only occur if $f$ is a sinusoid.
\end{itemize}
\end{theorem}

\begin{proof}  For (i)-(iii), see the cited references.
Now we prove (iv).  From part (iii) (or \eqref{pl}), we have $|f(z)| \leq A e^{2\lambda r}$ for $z$ in the disk $D(x_0, 2r)$.  The claim now follows from Jensen's formula.

    For (v), we can rescale $A=\lambda=1$, $x_0=0$, and $f(0)=1$.  If we write $f(x) = \cos(g(x))$ for a continuous $g(x)$ with $g(0)=0$ that is smooth away from the solutions to $f(x) = \pm 1$, then from (ii) and the chain rule we have $|g'(x)| \leq 1$ for all $x$ with $f(x) \neq \pm 1$.  Integrating, we obtain $|g(x)| \leq |x|$ for all $x$, giving the first part of (v).  If equality occurs for any non-zero $x$ in this region, then $f(z)$ equals $\cos(z)$ on some non-trivial interval, and hence identically by analytic continuation, giving the second part of (v); and (vi) is an immediate corollary of (v).
\end{proof}

\begin{remark} One can improve (iv) asymptotically by showing that the number of zeroes is at most $(1+o(1)) \frac{2\lambda r}{\pi}$ as $r \to \infty$ by the Cartwright--Levinson theorem; see e.g., \cite{levin}.  In the special case of trigonometric polynomials (\Cref{poly-ex}), the bound (i) is a very well known inequality of Bernstein \cite{bernstein}, the sharper bound (ii) was established independently by Szeg\"o \cite{szego} and Schaake and van der Corput \cite{vdc}, the bound (iii) was established by Bernstein \cite{bern3}, and the result (vi) was obtained by Riesz \cite{riesz}.
\end{remark}

A unified treatment of these results was given\footnote{See also \cite{rahman} for an alternate approach.} by Duffin and Schaeffer \cite{duffin}, based primarily on applying both Rouch\'e's theorem and the intermediate value theorem to functions of the form
$$ A \cos(\lambda(z-x_0)) - \alpha f_\eps(z)$$
for $x_0 \in\R$ and $\eps>0$, where $f_\eps$ is a slightly dampened and rescaled version of $f$ and $0 < \alpha < 1$, and comparing the information about zeroes produced by these two theorems.  This argument matches well with the intuition that elements of ${\mathcal B}^{\infty,\R}_\lambda$ are ``dominated'' by sinusoids in various senses.

\subsection{Local Bernstein theory}

At first glance, Bernstein theory seems like an inappropriate tool to study the behavior of ordinary polynomials (rather than trigonometric polynomials), since such polynomials grow far slower than exponentially at infinity, and also cannot be globally bounded on the real line unless they are constant.  However, as we shall see later, when studying high-degree real polynomials $P$ which enjoy reasonable upper and lower bounds on a moderately large interval, these polynomials exhibit Bernstein class (or sinusoidal) type behavior \emph{locally}, for instance on long thin rectangles of the form
\begin{equation}\label{riy-0}
 R(I, y_0) \coloneqq \{ x+iy: x \in I; -y_0 \leq y \leq y_0 \}
\end{equation}
or the upper half
\begin{equation}\label{riy}
 R^+(I, y_0) \coloneqq \{ x+iy: x \in I; 0 \leq y \leq y_0 \}
\end{equation}
of such rectangles, where $I$ is an interval where good upper and lower bounds on $P$ exist, and $y_0>0$ is much smaller than the length of $|I|$.  One precise formalization of this intuition is given in \Cref{prelim} below; another key family of examples to keep in mind are the Chebyshev polynomials, as depicted in \Cref{fig-cheby}.  Because of this phenomenon, it becomes of interest to localize Bernstein theory to rectangles such as \eqref{riy}.  Montel's theorem suggests that such localization ought to be possible at a qualitative level at least; but for applications it is important to obtain local results with quantitative error terms.

It turns out that the Rouch\'e theorem (and intermediate value theorem) arguments of Duffin and Schaeffer are very amenable to localization to such rectangles, though of course some error terms in the estimates must necessarily appear as a consequence.  In \Cref{boas-sec} we establish the following local version of \Cref{ds}:

\begin{theorem}[Local Bernstein theory]\label{main-2}  Let $f$ be a holomorphic function on a rectangle $R^+(I,y_0)$.  Let $A, \lambda, L > 0$, and suppose one has the following bounds:
\begin{itemize}
\item[(a)] (Lower edge) $f(x)$ is real with $|f(x)| \leq A$ for all $x \in I$.
\item[(b)] (Upper edge) $|f(x+iy_0)| \leq A e^{\lambda y_0}$ for all $x \in I$.
\item[(c)] (Vertical edges) $|f(x+iy)| \leq A e^{\lambda y + e^{\frac{\pi L}{4y_0}}}$ for all $x \in \partial I$ and $0 \leq y \leq y_0$.
\end{itemize}
Let $x \in I$ with $\mathrm{dist}(x, \partial I) \geq L$.
\begin{itemize}
\item[(i)]  (Local Bernstein inequality)  One has
\begin{equation}\label{fxa}
 |f'(x)| \leq A \lambda \left( 1 + O( e^{-\frac{\pi L}{4y_0}}) + O\left(\frac{1}{\lambda \min(y_0,L)}\right) \right).
\end{equation}
\item[(ii)]  (Local Boas inequality)  In fact, one has the stronger inequality
\begin{equation}\label{la}
 \left|f(x) + \frac{i f'(x)}{\lambda}\right| \leq A \left( 1 + O\left( e^{-\frac{\pi L}{4y_0}}\right) + O\left(\frac{1}{\lambda \min(y_0,L)}\right) \right).
\end{equation}
\item[(iii)]  (Local Duffin--Schaeffer inequality) For any $0 \leq y \leq y_0$, one has
\begin{equation}\label{lds}
 |f(x+iy)| \leq A \left( 1 + O( e^{-\frac{\pi L}{4y_0}}) \right) \cosh \left( \left( 1 + O\left(\frac{1}{\lambda \min(y_0,L)}\right) \right) \lambda y \right).
\end{equation}
\item[(iv)]  (Local linear growth of zeroes) For $0 < r < \min(y_0,L)/4$ and $f(x) \neq 0$, then the number of zeroes of $f$ in the disk $D(x,r)$ is at most $O( 1 + \lambda r + \log \frac{A}{|f(x)|})$.
\end{itemize}
\end{theorem}

\begin{remark}
The condition (c) is very mild in practice due to the double exponential bound on the right-hand side; this is closely related to the double exponential growth hypothesis in the Phragm\'en--Lindel\"of three-lines theorem.  The error terms in \eqref{fxa}, \eqref{la}, \eqref{lds} could be improved slightly with more effort, and local analogues of the other components (v), (vi) of \Cref{ds} could also be obtained, but the above theorem will be sufficient for our application below.
\end{remark}

\subsection{Application to a problem of Erd\H{o}s and Tur\'an on Lagrange interpolation}

Our main motivation for introducing local Bernstein theory is to make progress on some classical problems regarding Lagrange interpolation.
Given a sequence of distinct points (or \emph{nodes}) $x_1 < \dots <x_n$ in $[-1, 1]$ with $n \geq 1$, we have the \emph{Lagrange
interpolation formula}
\begin{equation}\label{interp}
Q(z) = \sum_{k=1}^n Q(x_k) \ell_k(z),
\end{equation}
for all polynomials $Q$ of degree at most $n-1$, where the \emph{Lagrange basis function} (or \emph{fundamental functions}) $\ell_k$ is the
degree $n-1$ polynomial that vanishes at all of the $x_i$ except for $x_k$, where it takes the value
1. Explicitly, we have
\begin{equation}\label{lk-def}
\ell_k(z) \coloneqq \prod_{i \neq k} \frac{z - x_i}{x_k - x_i}.
\end{equation}
If we view $x_1,\dots,x_n$ as the roots of the monic degree $n$ polynomial
\begin{equation}\label{poly-def}
    P(z) \coloneqq \prod_{i=1}^n (z - x_i),
\end{equation}
we can also write
\begin{equation}\label{lk-form}
 \ell_k(x) = \frac{P(x)}{P'(x_k) (x-x_k)}.
\end{equation}

\begin{remark}\label{res} If we insert \eqref{lk-form} into \eqref{interp}, we obtain the fact that the residues of $\frac{Q(z)}{P(z) (x-z)}$ sum to zero, which is also an immediate consequence of the residue theorem (on the Riemann sphere).
\end{remark}

\begin{example}[Chebyshev polynomial]\label{cheby} If
$$x_k = \cos \left( \frac{2n-2k+1}{2n} \pi \right)$$
for $k=1,\dots,n$, then $P$ becomes the \emph{monic Chebyshev polynomial} $P(x) = 2^{1-n} T_n(x)$ where $T_n$ is the $n^{\operatorname{th}}$ Chebyshev polynomial, defined by the formula $T_n(\cos \theta) = \cos(n \theta)$.  Heuristically, for the purposes of performing interpolation in $[-1,1]$, this choice of nodes is expected to make the Lagrange basis functions nearly as ``small'' as possible in various senses.  A useful heuristic to keep in mind is that these monic Chebyshev polynomials $P(x)$ behave locally like sinusoids of amplitude $2^{1-n}$ and frequency $\pi n \rho_{\mathrm{as}}(x_0)$ around any fixed point $x_0$ in the interior of $[-1,1]$, where $\rho_{\mathrm{as}}(x) \coloneqq \frac{1}{\pi \sqrt{1-x^2}}$ is the arcsine distribution; see \Cref{fig-cheby}.  This is in sharp contrast to the global behavior of such polynomials, which of course diverges to infinity at a polynomial rate rather than growing exponentially or oscillating sinusoidally.
\end{example}

\begin{figure}
    \centering
\includegraphics[width=0.75\textwidth]{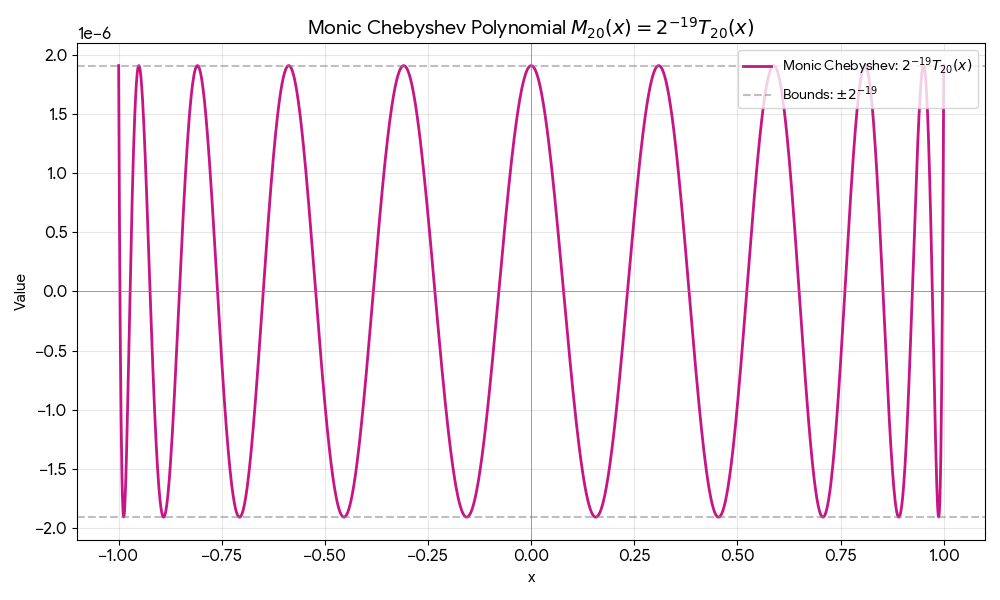}
    \caption{The monic Chebyshev polynomial $P(x) = 2^{1-n} T_n(x)$ with $n=20$.  Note the local sinusoidal behavior in the interior of the interval $[-1,1]$.  Not displayed is the rapid (and non-sinusoidal) growth of $P$ outside of this interval; see \Cref{fig-pot} for a depiction of that growth in (negative) log-scale. (Image generated by Gemini.)}
    \label{fig-cheby}
\end{figure}

The Lagrange interpolation formula \eqref{interp} can be viewed as describing a projection operator from the space $C([-1,1])$ to the space of polynomials of degree at most $n-1$.  The \emph{$C([-1,1])$ Lebesgue constant} is the operator norm of that projection, which can be easily seen to be
\begin{equation}\label{maxl}
 \sup_{x \in [-1,1]} \lambda(x)
\end{equation}
where $\lambda \colon \R \to \R$ is the
\emph{Lebesgue function}
\begin{equation}\label{lambda-def}
\lambda(x) := \sum_{k=1}^n |\ell_k(x)|.
\end{equation}
By \eqref{lk-form} we can also write this function as
\begin{equation}\label{lamax}
 \lambda(x) = \sum_{k=1}^n \frac{|P(x)|}{|P'(x_k)| |x-x_k|}.
\end{equation}
Informally, one can subdivide \eqref{lambda-def} or \eqref{lamax} into three components:
\begin{itemize}
    \item The \emph{macroscopic} contribution where $|x-x_k|$ is somewhat close to $1$ (e.g., $|x-x_k| \geq n^{-c}$ or $|x-x_k| \geq \log^{-C} n$ for some constants $c,C>0$);
    \item The \emph{microscopic} contribution where $|x-x_k|$ is somewhat close to $1/n$ (e.g., $|x-x_k| \leq n^{-1+c}$ or $|x-x_k| \leq \log^C n/n$ for some constants $c,C>0$); and
    \item The \emph{mesoscopic} contribution where $|x-x_k|$ is far from both $1/n$ and $1$ (e.g., $n^{-1+c} < |x-x_k| < n^{-c}$ or $\log^C n/n < |x-x_k| < \log^{-C} n$ for some constants $c,C>0$).
\end{itemize}
In practice, the mesoscopic scales generate the bulk of the contributions to the Lebesgue function, and are the easiest (and most ``universal'') contribution to control; but one also needs to understand the macroscopic and microscopic contributions to obtain the strongest error terms in the estimates.

The study of the $C([-1,1])$ Lebesgue constant \eqref{maxl} received considerable attention in the previous century.  As already observed by Runge \cite{runge}, these constants can become exponentially large even in apparently benign situations, such as equally spaced zeroes; see \cite[\S 2.1]{brutman} for further discussion.  Here we focus instead on lower bounds for this constant that are valid for all choices of $x_k \in [-1,1]$.  The bound
$$ \sup_{x \in [-1,1]} \lambda(x) \gg \log n$$
was proven by Faber \cite{faber}, with a simpler proof (giving the explicit lower bound of $\frac{1}{12} \log n$) given by Fej\'er \cite{fejer}.
Bernstein \cite{bernstein} asserted (without full proof details) the asymptotically optimal lower bound
\begin{equation}\label{bern}
    \sup_{x \in [-1,1]} \lambda(x) \geq \left( \frac{2}{\pi} - o(1) \right) \log n
\end{equation}
which was subsequently improved by Erd\H{o}s and Tur\'an \cite{et} to
$$ \sup_{x \in [-1,1]} \lambda(x) \geq \frac{2}{\pi} \log n - O(\log\log n).$$
Roughly speaking these arguments required understanding the mesoscopic contributions to $\lambda(x)$ only.  By obtaining further control on both macroscopic and microscopic contributions, Erd\H{o}s \cite{erdos} improved the bound further to
\begin{equation}\label{erdos-error}
\sup_{x \in [-1,1]} \lambda(x) \geq \frac{2}{\pi} \log n - O(1)
\end{equation}
for all large $n$.  Finally, Vertesi \cite[Theorem 3.1]{vertesi} showed that
$$ \sup_{x \in [-1,1]} \lambda(x) \geq \frac{2}{\pi} \log n + C + o(1)$$
where
$$ C \coloneqq \frac{2}{\pi} \left(\gamma + \log \frac{4}{\pi}\right) = 0.521251\dots \text{(\href{https://oeis.org/A243258}{OEIS A243258})}; $$
and (using \Cref{cheby}) also showed that this bound is optimal up to the $o(1)$ error;
we refer the reader to \cite{brutman} for further discussion. See also \Cref{fig-lebesgue}.
Thus the minimal value of \eqref{maxl} is determined asymptotically.  We remark that global Bernstein theory, as exemplified by \Cref{ds}, plays an important role in the proofs of these bounds (after making the standard substitution $x = \cos \theta$ to transform a polynomial into a trigonometric polynomial); in particular, the Bernstein inequality in \Cref{ds}(a) can be used to help lower bound the denominators in \eqref{lamax}.

\begin{figure}
    \centering
\includegraphics[width=0.75\textwidth]{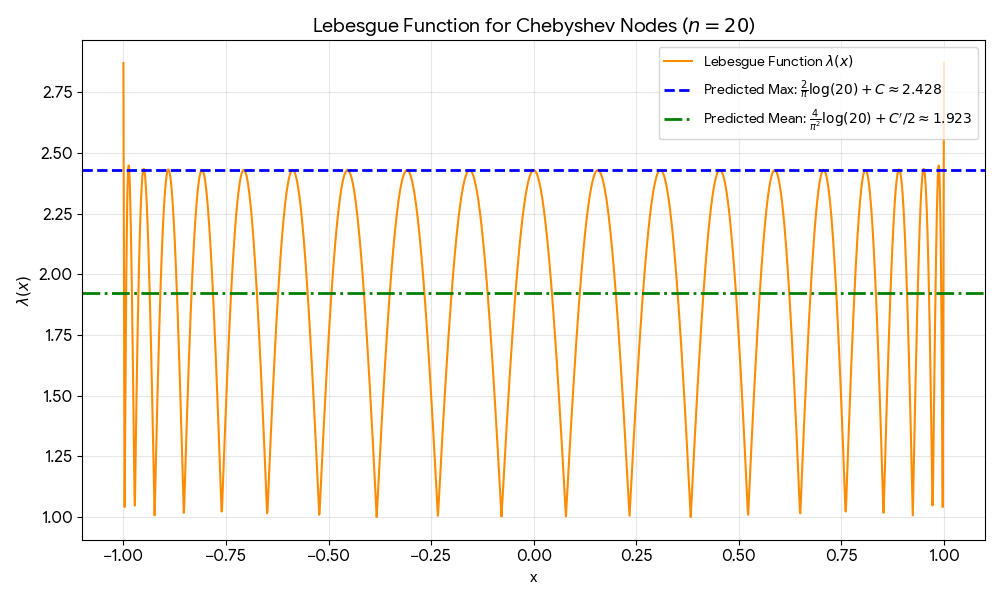}
    \caption{The Lebesgue function $\lambda(x)$ for the polynomial in \Cref{fig-cheby}, together with the predicted maximal value of $\frac{2}{\pi} \log n + C$, which holds up well in the bulk of $[-1,1]$ but becomes less accurate near the endpoints.  The predicted mean of $\frac{4}{\pi^2} \log n + \frac{C'}{2}$, which is smaller by a factor of about $\frac{2}{\pi}$, is also shown.  (Image generated by Gemini.)}
    \label{fig-lebesgue}
\end{figure}

The integral of the Lebesgue function $\lambda$ has also been studied. When the $x_1,\dots, x_n$ are the roots of the Chebyshev polynomial (\Cref{cheby}), it was shown in
\cite[Theorem 3.1]{bgt}, that
$$ \int_{-1}^1 \lambda(x)\ dx = \frac{8}{\pi^2} \log n + C' + o(1)$$
where
$$ C' \coloneqq \frac{8}{\pi^2} \Biggl( \gamma + \int_0^{\pi/2} \frac{\sin x}{x}\ dx - \int_0^{\pi/2} \frac{1-\cos x}{x}\ dx $$
$$ + \int_0^{\pi/2} \frac{\cot x (x-\sin x)}{(1-2x/\pi) x}\ dx\Biggr) = 1.417018\dots.$$
In particular, the mean value $\frac{4}{\pi^2} \log n + \frac{C'}{2} + o(1)$ of $\lambda$ is about $\frac{2}{\pi}$ times the sup norm, which is consistent with the fact that the $L^1$ mean of a sinusoid is $\frac{2}{\pi}$ times its sup norm; see \Cref{fig-lebesgue}.

In \cite{erd-cluj} it was conjectured that the Chebyshev nodes were asymptotically optimal for the integral of $\lambda$, so that one had the general lower bound
\begin{equation}\label{integral-conj}
    \int_{-1}^1 \lambda(x)\ dx \geq \frac{8}{\pi^2} \log n + C' + o(1)
\end{equation}
for arbitrary $-1 \leq x_1 < \dots < x_n \leq 1$.

Erd\H{o}s and Tur\'an asked \cite[p. 224]{et}, \cite{erdos-open}, \cite{various}, \cite[Problem 1153]{bloom} whether analogous lower bounds could be established for the smaller \emph{$C(I)$ Lebesgue constant}
\begin{equation}\label{maxl-loc}
 \sup_{x \in I} \lambda(x)
\end{equation}
where $I$ is a fixed non-trivial subinterval of $[-1,1]$.  In this direction, the lower bound of
$$  \int_I \lambda(x)\ dx \gg |I| \log n$$
was shown by Erd\H{o}s and Szabados \cite{esz} for $n$ sufficiently large depending on $I$ (see also a refinement in \cite{eszv}), which by the pigeonhole principle implies that
$$  \sup_{x \in I} \lambda(x) \gg \log n$$
for the same range of $n$. As recently as 1993, Erd\H{o}s  \cite{erdos-open} wrote (in our notation) that ``Probably
\begin{equation}\label{lln}
    \sup_{x \in I} \lambda(x) \geq \frac{2}{\pi} \log n - o(\log n)
\end{equation}
but we are very far\footnote{This result had been announced in \cite[(6)]{erd-cluj}, but in view of \cite{erdos-open} it appears that the proof was incomplete.} from being able to prove this.''

\subsection{Main application, and a trigonometric toy model}

As an application of \Cref{main-2}, we will be able to answer this question of Erd\H{o}s and Tur\'an in the affirmative (with a stronger error term), while also obtaining a weak form of the integral lower bound \eqref{integral-conj} conjectured in \cite{erd-cluj}:

\begin{theorem}[Main theorem]\label{main}  Let $I \subset [-1,1]$ be a fixed interval.  Let $n$ be sufficiently large, and let $x_1 < \dots < x_n$ be distinct points in $[-1,1]$.
    \begin{itemize}
  \item[(i)] (Sup norm bound) We have
        \begin{equation}\label{erdos-conj}
     \sup_{x \in I} \lambda(x) \geq \frac{2}{\pi} \log n - O(1).
\end{equation}
\item[(ii)] (Integral bound) We have
\begin{equation}\label{erdos-conj-2}
     \int_I \lambda(x)\ dx \geq \frac{4|I|}{\pi^2} \log n - o(\log n).
\end{equation}
In particular,
$$ \int_{-1}^1 \lambda(x)\ dx \geq \frac{8}{\pi^2} \log n - o(\log n).$$
    \end{itemize}
Here and in the sequel the implied constant in the $O()$ notation is allowed to depend on $I$.
\end{theorem}

By combining \Cref{main}(i) with a standard Baire category argument, we can get very close to resolving a question of Erd\H{o}s \cite[p. 68]{erd-cluj}, \cite[Problem 1132]{bloom}:

\begin{corollary}[Lower bound at a point]\label{div}  For each $n$ let $-1 \leq x^{(n)}_1 < \dots < x^{(n)}_n \leq 1$ be distinct points, and let $\lambda^{(n)}$ be the associated Lebesgue function, and let $\omega \colon \N \to \R^+$ be any function that goes to infinity as $n \to \infty$ (e.g., $\omega(n) = \log\log\log n$).  Then there exists a dense set of points $x^* \in [-1,1]$ such that $\lambda^{(n)}(x^*) \geq \frac{2}{\pi} \log n - \omega(n)$ for infinitely many $n$.
\end{corollary}

\begin{proof}  Suppose this claim failed for some sufficiently large constant $C$, then there is an interval $I_0 \subset [-1,1]$ such that each $x^* \in I_0$ is contained in a set of the form
    $$ E_{n_0} \coloneqq \left \{ x^* \in [-1,1]: \lambda^{(n)}(x^*) \leq \frac{2}{\pi} \log n - \omega(n) \text{ for all } n \geq n_0 \right \}.$$
These sets are all closed, hence by the Baire category theorem, one of these sets $E_{n_0}$ must contain a non-trivial interval $I \subset I_0 \subset [-1,1]$, so in particular
$$ \sup_{x \in I} \lambda(x) \leq \frac{2}{\pi} \log n - \omega(n)$$
for all sufficiently large $n$.  But this contradicts \Cref{main}(i).
\end{proof}

\begin{remark}\label{rem-baire}  The above argument in fact shows that the set of such $x^*$ is comeager in $[-1,1]$.
    The problem in \cite[p. 68]{erd-cluj}, \cite[Problem 1132]{bloom} seeks to replace the slowly growing function $\omega(n)$ by a constant independent of $n$ (though it is left unspecified whether the constant is permitted to depend on $x_*$).  We do not see a quick way to get this stronger conclusion purely\footnote{The main difficulty here is that the implied constant in the $O(1)$ error that is produced by our arguments will depend on $I$.  It is the author's tentative belief that this dependency on $I$ in the error term cannot be entirely eliminated, and that the final resolution to this problem of Erd\H{o}s cannot be achieved purely by a Baire category approach.  In \cite[p. 68]{erd-cluj}, \cite[Problem 1132]{bloom} it is also asked whether ``dense subset'' can be upgraded to ``full measure subset''; we have nothing new to say in this direction, as it seems to require establishing good lower bounds on $\lambda(x)$ not just for one point in an interval $I$, but on a uniformly positive measure subset of $I$.} from \Cref{main}(i); it is noted in \cite[p. 68]{erd-cluj} that ``this problem does not seem to be easy''.  In these references it is also conjectured that for almost all $x^* \in [-1,1]$ one has $\lambda^{(n)}(x^*) \geq \frac{2}{\pi} \log n - o(\log n)$ for infinitely many $n$, but again we see no quick way to achieve this.
\end{remark}

If one takes the ansatz (inspired by \Cref{cheby}) that the polynomial $P$ behaves like a trigonometric polynomial on $I$ (with period comparable to $1/n$), applies \eqref{lamax}, and focuses on the mesoscopic contributions to \eqref{erdos-conj}, \eqref{erdos-conj-2}, then after some routine rescaling one can extract the following simplified toy model for \Cref{main}:

\begin{theorem}[Trigonometric toy model]\label{toy}
Let $P \colon \R \to \R$ be a trigonometric polynomial of degree $n$, with $2n$ distinct zeroes $x_1,\dots,x_{2n}$ in $[0,2\pi)$.
\begin{itemize}
    \item[(i)] (Sup norm bound)  One has
    \begin{equation}\label{notr}
        \sup_{x \in [0,2\pi)} \sum_{k=1}^{2n} \frac{|P(x)|}{|P'(x_k)|} \geq 2.
    \end{equation}
    \item[(ii)]  (Integral bound)  One has
    \begin{equation}\label{integral}
     \int_0^{2\pi} \sum_{k=1}^{2n} \frac{|P(x)|}{|P'(x_k)|}\ dx \geq 8.
    \end{equation}
\end{itemize}
\end{theorem}

Note that the example of a sinusoid $P(x) = A \cos(n(x-x_0))$ shows that the bounds in \Cref{toy} are sharp.  A naive application of \eqref{integral} and the pigeonhole principle would give the weaker lower bound $\frac{2}{\pi} \times 2$ for \eqref{notr}; the loss of $\frac{2}{\pi}$ corresponds exactly to aforementioned ratio between the $L^1$ mean and $L^\infty$ norm of a sinusoid.

\Cref{toy}(i) is an immediate consequence of the Bernstein inequality (\Cref{ds}(i)) applied to the trigonometric polynomial $P$.  This already explains why (global or local) Bernstein theory is so relevant to the proof of results such as \Cref{main}(i), although to get the strong error term in that result we will also need a separate analysis of the microscopic and macroscopic contributions.  However, there does not seem to be a similarly quick derivation of \Cref{main}(ii) from known results in the literature.

The left-hand side of \eqref{integral} factors as the product of $\int_0^{2\pi} |P(x)|\ dx$ and $\sum_{k=1}^{2n} \frac{1}{|P'(x_k)|}$.  After some experimentation using the tool \emph{AlphaEvolve}, the author was led to conjecture a proof of  \eqref{integral} by separately lower bounding each of these two factors.  The first of these conjectures was then proven by ChatGPT, and the author was able to prove the second, thus giving a complete proof of \Cref{toy}.  More precisely, in \Cref{trig-sec} we will show

\begin{lemma}\label{trig-lem}  Let $P \colon \R \to \R$ be a trigonometric polynomial \eqref{trig}.
\begin{itemize}
    \item[(i)] ($L^1$ lower bound) one has
$$ \int_0^{2\pi} |P(x)|\ dx \geq 4 |a_n+ib_n|.$$
    \item[(ii)] (Lower bound on reciprocal magnitudes of derivatives)  If $P$ has $2n$ distinct zeroes $x_1,\dots,x_{2n}$ in $[0,2\pi)$, then one has
    $$ \sum_{k=1}^{2n} \frac{1}{|P'(x_k)|} \geq \frac{2}{|a_n+ib_n|}.$$
\end{itemize}
\end{lemma}

Multiplying the bounds in \Cref{trig-lem} yields \eqref{integral}.  Again, one can check that equality is attained in both inequalities when $P$ is a sinusoid.  The logical dependencies used to prove \Cref{toy} are summarized in \Cref{fig:trig}.

\begin{figure}
\centering
\begin{tikzpicture}[
    node distance=1cm and 1cm,
    box/.style={
        rectangle,
        draw,
        minimum width=3cm,
        minimum height=1cm,
        align=center
    },
    arrow/.style={
        -{Latex[length=3mm]},
        thick
    }
]

\node[box] (loc-bern) {\Cref{ds}};
\node[below =of loc-bern] (grid) {};
\node[right=of loc-bern] (prelim) {\qquad\qquad\qquad\qquad\qquad};
\node[box, below =of prelim] (new) {\Cref{trig-lem}};
\node[box, below =of grid] (parti) {\Cref{toy}(i)};
\node[box, below =of new] (partii) {\Cref{toy}(ii)};

\draw[arrow] (loc-bern) -- (parti);
\draw[arrow] (new) -- (partii);
\end{tikzpicture}
\caption{The logical dependencies between the main results of this paper involving trigonometric polynomials (or functions of global exponential type). The spacing here is chosen to be consistent with that in \Cref{fig:flow} below.}
\label{fig:trig}
\end{figure}
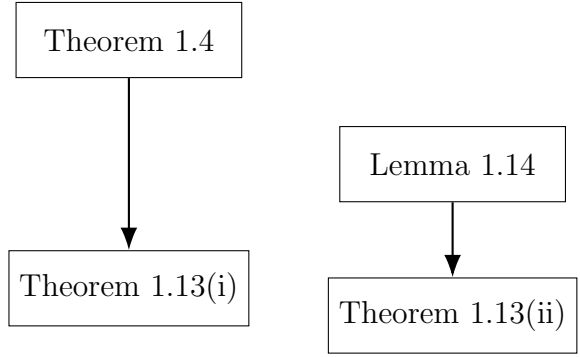

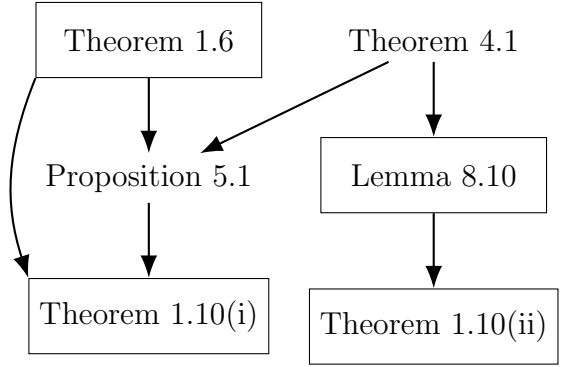
\begin{figure}
\centering
\begin{tikzpicture}[
    node distance=1cm and 1cm,
    box/.style={
        rectangle,
        draw,
        minimum width=3cm,
        minimum height=1cm,
        align=center
    },
    arrow/.style={
        -{Latex[length=3mm]},
        thick
    }
]

\node[box] (loc-bern) {\Cref{main-2}};
\node[below =of loc-bern] (grid) {\Cref{good-grid}};
\node[right=of loc-bern] (prelim) {\Cref{prelim}};
\node[box, below =of grid] (parti) {\Cref{main}(i)};
\node[box, below =of prelim] (new) {\Cref{newtrig}};
\node[box, below =of new] (partii) {\Cref{main}(ii)};

\draw[arrow] (loc-bern) -- (grid);
\draw[arrow, bend right = 20] (loc-bern.south west) to (parti.north west);
\draw[arrow] (prelim) -- (grid);
\draw[arrow] (prelim) -- (new);
\draw[arrow] (new) -- (partii);
\draw[arrow] (grid) -- (parti);
\end{tikzpicture}
\caption{The logical dependencies between the main results of this paper involving polynomials (or functions of local exponential type).  The results in boxes are analogous to the corresponding results in \Cref{fig:trig}. Additional dependencies involving other propositions and lemmas are omitted to reduce clutter.}
\label{fig:flow}
\end{figure}

\subsection{Proof sketch for the sup norm bound}

The logical dependencies used to prove \Cref{main} are analogous to (but more complicated than) those used to establish \Cref{toy}, and are summarized in \Cref{fig:flow}.

We are now ready to discuss the proof of \Cref{main}(i).  Suppose that \eqref{erdos-conj} fails.
We introduce the \emph{empirical measure}
\begin{equation}\label{empirical}
 \mu = \mu^{(n)} \coloneqq \frac{1}{n} \sum_{i=1}^n \delta_{x_i},
\end{equation}
where $\delta_x$ is the Dirac mass at $x$.  This is a probability measure supported on $[-1,1]$, and it generates a \emph{logarithmic potential}
\begin{equation}\label{umu-def}
 U_\mu(z) \coloneqq \int_\R \log \frac{1}{|z-x|}\,d\mu(x) = \frac{1}{n} \sum_{i=1}^n \log \frac{1}{|z-x_i|} = \frac{1}{n} \log \frac{1}{|P(z)|}.
\end{equation}
This potential is harmonic in the upper half-plane.  Informally, the behavior of $U_\mu(x+i\eta)$ captures the macroscopic (resp. mesoscopic, microscopic) behavior of $P$ if $\eta$ is comparable to $1$ (resp. between $1/n$ and $1$, comparable to $1/n$).

From \eqref{umu-def} and the failure of \eqref{erdos-conj} one can obtain good control on $U_\mu$ in the interior of $I$; indeed, it is close to a constant on average on this interval\footnote{We thank Nat Sothanaphan \cite{gpt} (using GPT) for supplying a rigorously proven version of this claim, which we reproduce in \Cref{prelim}(i).}.  Above and below this interval, $U_\mu$ turns out to behave roughly linearly, so that the polynomial $P$ behaves locally as if it were of exponential type, allowing \Cref{main-2} to be utilized.  Among other things, this allows one to assign an amplitude $A(x)$ and frequency $n\pi \rho(x)$ to each location $x$ in $I$, with $P$ being ``dominated'' in various senses by a sinusoid with those parameters; see \Cref{good-grid}.   This already gives good control on the macroscopic portion of the Lebesgue function \eqref{lamax} at some well-chosen location $x_*$: see \Cref{gp}.  The mesoscopic contributions are also easy to control from the previous steps; the main difficulty is to understand the microscopic contributions, and to in particular to avoid having an unexpectedly small number of nodes around $x_*$ at some microscopic scale $\frac{t_0}{n\rho(x_*)}$ for some $1 \ll t_0 \ll n^{0.01}$.  To eliminate this scenario we use a combination of both the global and local Bernstein theory, in which one applies the residue theorem (in the spirit of the residue theorem interpretation of \eqref{interp}) to a blended function which behaves locally like a rescaled version of $P$, but globally like a sinusoid (damped by several powers of a rescaled sinc function to keep certain boundary terms under control); see \Cref{micro-sec}.  This blended function technique is adapted from the arguments of Erd\H{o}s \cite{erdos} (where the role of the sinusoid is played by a Chebyshev polynomial).

\subsection{Proof sketch for the integral bound}

Now we discuss the proof of \Cref{main}(ii), the details of which we give in \Cref{integral-sec}.

We can rewrite the goal \eqref{erdos-conj-2} using \eqref{lamax} as
$$ \int_I \sum_{k: x_k \in I} \frac{|P(x)|}{|P'(x_k)| |x-x_k|}\ dx \geq (4-o(1)) \frac{|I|}{\pi^2} \log n.$$
By a Whitney-type decomposition of the $\frac{1}{|x-x_k|}$ kernel, it essentially suffices to obtain bounds of the form
$$ \int_{J_0} \sum_{k: x_k \in J_0} \frac{|P(x)|}{|P'(x_k)|}\ dx \geq (8-o(1)) \left( \frac{|J_0|}{2\pi} \right)^2$$
for various mesoscopic intervals $J_0$ in $I$ (cf. \eqref{integral}).  By localizing further to microscopic intervals $K_j$ on which $P$ behaves locally as if it were in a Bernstein class, the task is now to lower bound the quantities $\int_\R |P(x)| \varphi_j(x)$ and $\sum_k \frac{1}{|P'(x_k)|} \varphi_j(x_k)$, where $\varphi_j$ is a bump function adapted to $K_j$; see \Cref{newtrig} for the precise lower bounds we will establish.  As indicated by \Cref{fig:trig} and \Cref{fig:flow}, our proof strategy for this lemma will be motivated by the solution of the trigonometric toy model \Cref{toy} via \Cref{trig-lem}, which we give in \Cref{trig-sec}.

The polynomial $P(x)$ is somewhat poorly controlled on the real line.  However, if one shifts upwards to $P(x+i\eta)$ for some well-chosen mesoscopic scale $\eta$ (in our arguments we will take $n^{\eps^3/2-1} \leq \eta \leq n^{\eps^3-1}$ for some small $\eps>0$), then one can establish sinusoidal type behavior: one has a polar representation $P(x+i\eta) = |P(x+i\eta)| e^{i \theta(x)}$ where $\theta$ is a smooth function with derivative comparable to $n$.  The expectation is then that $P(x)$ also oscillates with this phase ``on the average'', and the roots $x_k$ of $P$ are roughly located at the points where $e^{i\theta(x)}$ is $\pm 1$ (with the sign here corresponding to the sign of $P'(x_k)$).  With this heuristic, it becomes reasonable to use the lower bounds
\begin{equation}\label{ir}
\int_\R |P(x)| \varphi_j(x) \geq \int_\R P(x) \varphi_j(x) \mathrm{sgn}(\cos(\theta(x)))\ dx
\end{equation}
and
\begin{equation}\label{ir-2}
    \sum_k \frac{1}{|P'(x_k)|} \varphi_j(x_k) \geq \mathrm{Re} \sum_k \frac{1}{P'(x_k)} \varphi_j(x_k) e^{i\theta(x_k)}.
\end{equation}
It turns out that the right-hand side of \eqref{ir} can be adequately controlled by Fourier expanding the signum function and then shifting the contour by $i\eta$, using an approximately holomorphic function $e^{\pi n y \rho(x_J)} e^{-i\theta(x)} \varphi_j(x)$ as a weight.  The right-hand side of \eqref{ir-2} can be similarly controlled by applying the (weighted) residue theorem to $\frac{1}{P(z)}$ integrated against a complementary weight $e^{-\pi n y \rho(x_J)} e^{i\theta(x)} \varphi_j(x)$.  Neglecting some manageable error terms, one can then control these right-hand sides in terms of weighted integrals of $|P(x+i\eta)|$ and $\frac{1}{|P(x+i\eta)|}$ respectively (cf., the factors $|a_n+ib_n|$ and $\frac{1}{|a_n+ib_n|}$ in \Cref{toy}, as well as \eqref{P-approx}).  The product of such integrals can then be lower-bounded using the Cauchy--Schwarz inequality to obtain the required estimate.

\begin{remark}
It is plausible that the methods of \Cref{main}(i) and \Cref{main}(ii) can be combined to improve the error term in \eqref{erdos-conj-2} to $O(1)$, but this appears to be somewhat complicated technically (requiring one to control the microscopic and macroscopic contributions to the integral in \eqref{integral}), and we will not attempt to do so here.
\end{remark}

\subsection{Acknowledgments and AI tool disclosure}

The author learned about the problem \eqref{lln} from the Erd\H{o}s problem site \cite{bloom}. The author was supported by the James and Carol Collins Chair, the Mathematical Analysis \& Application Research Fund, and by NSF grant DMS-2347850, and is particularly grateful to recent donors to the Research Fund. We thank Bum Jun Kim for corrections.

Many of the references here were located by ChatGPT DeepResearch, Gemini DeepResearch, and Claude.  The proof of \Cref{tpl} is based on suggestions from ChatGPT Pro, Gemini Pro, and Claude.  GPT was used separately by Nat Sothanaphan, Aron Bhalla, and myself to identify corrections in an earlier version of the manuscript.  As already mentioned, AlphaEvolve and ChatGPT Pro were used to discover the proof of \Cref{toy} via \Cref{trig-lem}.  Github Copilot was used to perform autocompletion of text, Claude Code was used to perform routine typesetting, and Gemini was used to generate most of the plots in this paper; but besides those tool usages, the text in this paper was human-written.

\subsection{Notation} We use $X \ll Y$, $Y \gg X$, or $X = O(Y)$ to denote the estimate $|X| \leq CY$ for some constant $C>0$, and write $X \asymp Y$ for $X \ll Y \ll X$. If we need the implied constant $C$ to depend on parameters, we indicate this by subscripts, thus for instance $O_{C_0}(1)$ denotes a quantity bounded in magnitude by $C_{C_0}$ for some $C_{C_0}$ depending only on $C_0$.  However, in many cases we will explicitly omit the dependence on some very frequently occurring parameters (such as the interval $I$) to reduce clutter.

We will designate some mathematical objects (e.g., an interval $I$) as ``fixed'' with respect to the asymptotic parameter $n$.  We then use $o(X)$ to denote any quantity bounded in magnitude by $c(n) X$, where $c(n)$ is an expression depending on $n$ and the fixed quantities that goes to zero as $n \to \infty$ (holding the other quantities fixed).

If $I,J$ are intervals, we use $J \Subset I$ to denote the claim that the closure of $J$ is contained in the interior of $I$.  We use $\partial I$ to denote the boundary of $I$, and $|I|$ to denote its length.  All intervals will be understood to have positive finite length.  If $c>0$, we use $cI$ to denote the interval with the same midpoint as $I$ but with length $c|I|$.

\section{Some Fourier, complex, and harmonic analysis}

We record here some standard results from Fourier, complex, and harmonic analysis that we use in the paper.

\subsection{The square wave}

We first recall the standard Fourier expansion of a square wave:

\begin{lemma}[Fourier expansion of square wave]\label{lem-square}  One has the Fourier expansions
    \begin{equation}\label{fourier-sum-sine}
 \mathrm{sgn}(\sin x) = \frac{4}{\pi} \sum_{m \text{ odd}} \frac{\sin(mx)}{m}
\end{equation}
and
    \begin{equation}\label{fourier-sum}
 \mathrm{sgn}(\cos x) = \frac{4}{\pi} \sum_{m \text{ odd}} (-1)^{(m-1)/2} \frac{\cos(mx)}{m}
\end{equation}
where the Fourier expansion converges in $L^2([0,2\pi))$, and $m$ ranges over odd natural numbers.  In particular, taking Fej\'er sums in \eqref{fourier-sum}, we have
\begin{equation}\label{fourier-sum-fejer}
 \left|\frac{4}{\pi} \sum_{m \text{ odd}} \left(1-\frac{m}{M}\right)_+ \frac{\sin(mx)}{m}\right|, \left|\frac{4}{\pi} \sum_{m \text{ odd}} \left(1-\frac{m}{M}\right)_+ (-1)^{(m-1)/2} \frac{\cos(mx)}{m}\right| \leq 1
\end{equation}
for any integer $M \geq 1$ and real $x$. (See \Cref{fig:fejer}.)
\end{lemma}

This lemma can be proven by a routine calculation of the Fourier coefficients of $\mathrm{sgn}(\sin x)$ and $\mathrm{sgn}(\cos x)$, and is omitted; see for instance \cite[Exercise 11.19]{apostol} or \cite[1.442.3-4]{gradshteyn} for essentially the same Fourier series expansion.

\begin{figure}
\centering
\includegraphics[width=0.75\textwidth]{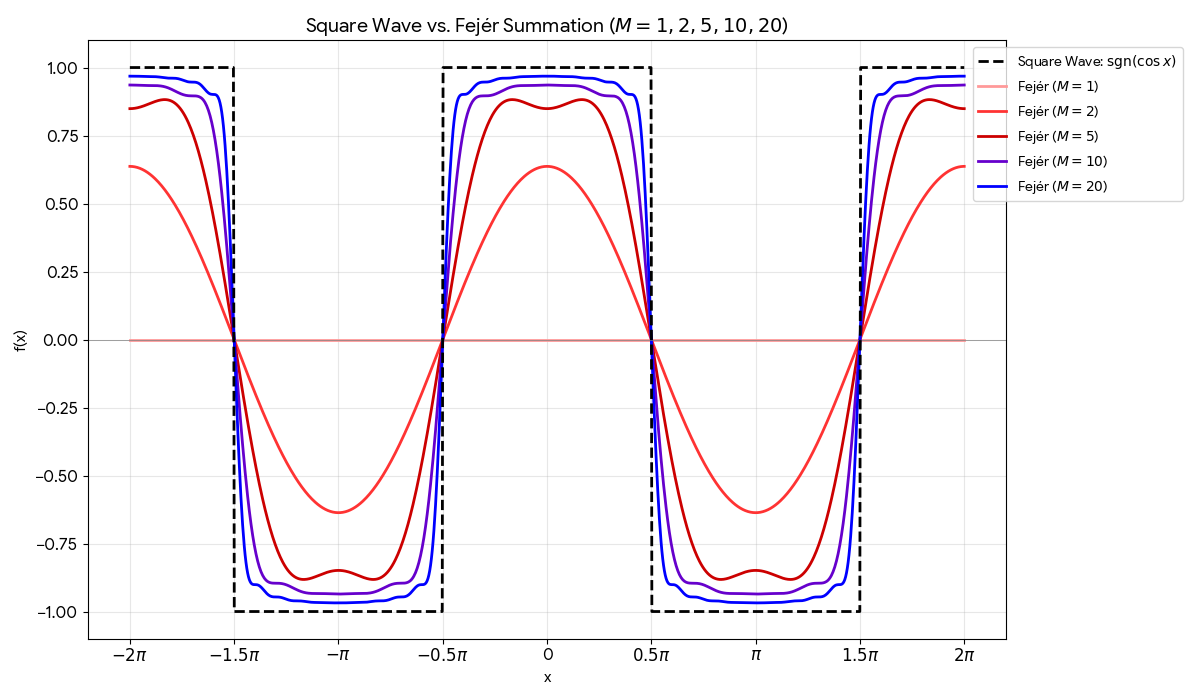}
\caption{The square wave $\mathrm{sgn}(\cos x)$, together with the Fej\'er sum approximant \eqref{fourier-sum-fejer} with $M=1, 2, 5,10, 20$.  Note how the approximant, being the convolution of the square wave with a non-negative approximation to the identity (the Fej\'er kernel), stays bounded by $1$ in magnitude, avoiding the Gibbs phenomenon.  (Image generated by Gemini.)}
\label{fig:fejer}
\end{figure}

\subsection{The residue theorem}

A fundamental tool in complex analysis is the residue theorem, one version of which we give here:

\begin{theorem}[Residue theorem]\label{residue}  Let $f$ be a meromorphic function on a neighborhood of a polygon $\Omega$, with no poles on the boundary $\partial \Omega$.  Then one has
    $$ \frac{1}{2\pi i} \int_{\partial \Omega} f(z)\ dz = \sum_{z_0 \in \Omega} \mathrm{Res}(f, z_0)$$
where $z_0$ ranges over poles of $f$ in $\Omega$, $\mathrm{Res}(f, z_0)$ is the residue of $f$ at $z_0$, and $\partial \Omega$ is traversed once anticlockwise.
\end{theorem}

In our main arguments, it will be convenient to work with a generalization of this theorem in which we insert a smooth weight $w$ that is only \emph{approximately} holomorphic in the sense that the \emph{Wirtinger derivative}
$$ \partial_{\bar z} w \coloneqq \frac{1}{2} \left( \partial_x + i \partial_y \right) w$$
is small, but not necessarily zero.  Of course, the Cauchy--Riemann equations assert that this derivative vanishes if and only if $w$ is genuinely holomorphic.

\begin{theorem}[Weighted residue theorem]\label{residue-weight}  Let $f$ be a meromorphic function on a neighborhood of a polygon $\Omega$, with no poles on $\partial \Omega$ and only simple poles in the interior of $\Omega$.  Let $w$ be a smooth (but not necessarily holomorphic) complex-valued function on a neighborhood of $\Omega$.  Then
    $$ \frac{1}{2\pi i} \int_{\partial \Omega} f(z) w(z)\ dz = \sum_{z_0 \in \Omega} \mathrm{Res}(f, z_0) w(z_0) + \frac{1}{\pi} \int_\Omega f(z) \partial_{\bar z} w(z)\ dA(z)$$
where $dA$ is area measure.
\end{theorem}

\begin{proof}  As $f$ only has simple poles, $\partial_{\bar z} (fw)$ is locally integrable in $\Omega$.  Excising small disks $D(z_0,\eps)$ around each pole $z_0$ and using Stokes' theorem in the remaining region, we can write
$$ \frac{1}{\pi} \int_\Omega \partial_{\bar z} (f(z) w(z))\ dA(z) = \frac{1}{2\pi i} \int_{\partial \Omega} f(z) w(z)\ dz - \lim_{\eps \to 0} \sum_{z_0 \in\Omega} \frac{1}{2\pi i} \int_{\partial D(z_0,\eps)}  f(z) w(z)\ dz,$$
where the circles $\partial D(z_0,\eps)$ are traversed once anticlockwise.

From the Leibniz rule and the Cauchy--Riemann equations we have
$$ \partial_{\bar z} (f(z) w(z)) = f(z) \partial_{\bar z} w(z)$$
away from the poles. For $z \in D(z_0,\eps)$, we can approximate $w(z) = w(z_0) + O(\eps)$ and $f(z) = \frac{\mathrm{Res}(f,z_0)}{z-z_0} + O(1)$, and the claim then follows by standard calculations.
\end{proof}

\subsection{Some inequalities on trigonometric polynomials}\label{trig-sec}

As a quick application of the above tools, we can now establish \Cref{trig-lem}.

We first prove \Cref{trig-lem}(i).  By\footnote{This argument was provided by ChatGPT Pro.} translating $P$ if necessary, we may assume that $b_n=0$, and then we can rescale $a_n=1$.  Thus we have
\begin{equation}\label{PQ}
    P(x) = \cos nx + Q(x)
\end{equation}
for some trigonometric polynomial $Q$ of degree at most $n-1$, and our task is now to show that $\int_0^{2\pi} |P(x)|\ dx \geq 4$.  On the other hand, from \eqref{fourier-sum-fejer} and the triangle inequality we have
$$ \int_0^{2\pi} \frac{4}{\pi} \sum_{m \text{ odd}} \left(1-\frac{m}{M}\right)_+ (-1)^{(m-1)/2} \frac{\cos(mnx)}{m} P(x)\ dx \leq \int_0^{2\pi} |P(x)|\ dx$$
for any integer $M \geq 1$.  By \eqref{PQ} and Fourier orthogonality\footnote{Alternatively, one can write $\cos(mnx) = \frac{1}{2} \mathrm{Re} e^{imnx}$ and shift the contour of integration upwards towards $+i\infty$.  This will be closer in spirit to the arguments we will develop below to establish \Cref{main}(ii).}, the left-hand side is $4(1-\frac{1}{M})$.  Sending $M \to \infty$ gives the claim.

Now we prove \Cref{trig-lem}(ii).  Again we can translate and normalize so that $b_n=0$ and $a_n=1$, so that $P$ takes the form \eqref{PQ}. By viewing $P$ as a degree $2n$ polynomial in $e^{ix}$, divided by $e^{inx}$, we see that all the zeroes of $P$ are on the real line. By the triangle inequality, we have
$$ \sum_{k=1}^{2n} \frac{1}{|P'(x_k)|} \geq -\sum_{k=1}^{2n} \frac{\sin(nx_k)}{P'(x_k)}  = -\Im \sum_{k=1}^{2n} \frac{e^{inx_k}}{P'(x_k)} .$$
The summands $-\frac{e^{inx_k}}{P'(x_k)}$ are the residues of $-\frac{e^{inz}}{P(z)}$.  By \Cref{residue} and the $2\pi$-periodicity of $P$, one can write the right-hand side as
$$\Im \frac{1}{2\pi i} \left( \int_0^{2\pi} \frac{e^{in(x+iT)}}{P(x+iT)}\ dx - \int_0^{2\pi} \frac{e^{in(x-iT)}}{P(x-iT)}\ dx \right)$$
for any $T>0$.  From \eqref{PQ} we have
$$ P(x \pm iT) = \frac{1}{2} e^{nT \mp inx} (1 + O( e^{-T} ))$$
(cf., \eqref{P-approx}) from which we see that the integrals in the above expression converge to $0$ and $4\pi$ respectively as $T \to +\infty$, giving the claim.

\subsection{Harmonic measure}

Given a planar domain $\Omega$ and a point $z \in \Omega$, the harmonic measure $\omega_z$ is the probability measure on the boundary $\partial \Omega$ describing the distribution of the first exit point of a Brownian motion started at $z$, so that
\begin{equation}\label{harmonic}
    u(z) = \int_{\partial \Omega} u(v)\ d\omega_z(v)
\end{equation}
when $u$ is harmonic on $\Omega$ (and continuous up to the boundary), and
\begin{equation}\label{subharmonic}
    u(z) \leq \int_{\partial \Omega} u(v)\ d\omega_z(v)
\end{equation}
if $u$ is subharmonic, assuming reasonable growth conditions at infinity in the case that $\Omega$ is unbounded.

For instance, if $\Omega$ is the upper half-plane, the probability distribution $\mu_{x_0+i\eta}$ is given by the Poisson kernel
\begin{equation}\label{poisson-mes}
 \mu_{x_0+i\eta} = \Poisson_\eta(x_0-x)\ dx,
\end{equation}
where $\Poisson_\eta$ is the Poisson kernel
\begin{equation}\label{poisson-kernel}
 \Poisson_\eta(x) \coloneqq \frac{\eta}{\pi(x^2 + \eta^2)}.
\end{equation}
We therefore have the representation
\begin{equation}\label{poisson}
 u(x_0+i\eta) = \int_\R \Poisson_\eta(x_0-x) u(x)\ dx
\end{equation}
if $u$ is harmonic on the upper half-plane, continuous up to the boundary, and grows slower than linearly with $\int_\R \frac{|u(x)|}{1+x^2}\ dx < \infty$ (e.g., one has $u(z) \ll \log |z|$ for all large $z$).  Indeed, one can easily check that the right-hand side of \eqref{poisson} is harmonic in the upper half-plane, continuous up to the boundary, and equals $u(x)$ on the real line, so by reflection the difference between the two sides of \eqref{poisson} is harmonic on all of $\C$, vanishes on the real line, and grows slower than linearly, hence vanishes entirely.

If in addition $u$ is square-integrable on the real line, we recall the classical Littlewood--Paley $L^2$ identity
\begin{equation}\label{littlewood-paley}
     \int_0^\infty \int_\R |\nabla u(x+iy)|^2\ y dx dy = \frac{1}{2} \int_{\R} u(x)^2\ dx
\end{equation}
where we write $|\nabla u|^2$ as shorthand for $|\partial_x u|^2 + |\partial_y u|^2$.  See for instance \cite[\S IV.1.2]{stein} for a proof.

We will need some estimates of harmonic measure on a rectangle $R^+(I,y_0)$.

\begin{lemma}[Harmonic measure]\label{harm-lemma}  Let $R^+(I,y_0)$ be a rectangle of the form \eqref{riy}, and let $x_0+i\eta$ be a point in the interior of this rectangle.
    \begin{itemize}
        \item[(i)] On the lower edge $\{ x: x \in I\}$ of the rectangle, the harmonic measure $\omega_{x_0+i\eta}$ is bounded by the Poisson kernel \eqref{poisson-mes}.
        \item[(ii)] The mass that the harmonic measure $\omega_{x_0+i\eta}$ assigns to the lower edge is at most $1 - \frac{\eta}{y_0}$, and the mass it assigns to the upper edge is at most $\frac{\eta}{y_0}$.
        \item[(iii)] The mass that the harmonic measure $\omega_{x_0+i\eta}$ assigns to the left and right edges is $O(e^{-\pi \mathrm{dist}(x_0, \partial I) / y_0})$.
    \end{itemize}
    See Figure \ref{fig-rect}.
\end{lemma}

In applications we will apply this lemma in situations where the height $y_0$ of the rectangle will be shorter than the length $|I|$, and $x_0+i\eta$ is close to the bottom edge and far from the left and right edges.  In such cases, the lower edge of the rectangle will absorb almost all of the harmonic measure.

\begin{figure}
    \centering
    \includegraphics[width=0.75\textwidth]{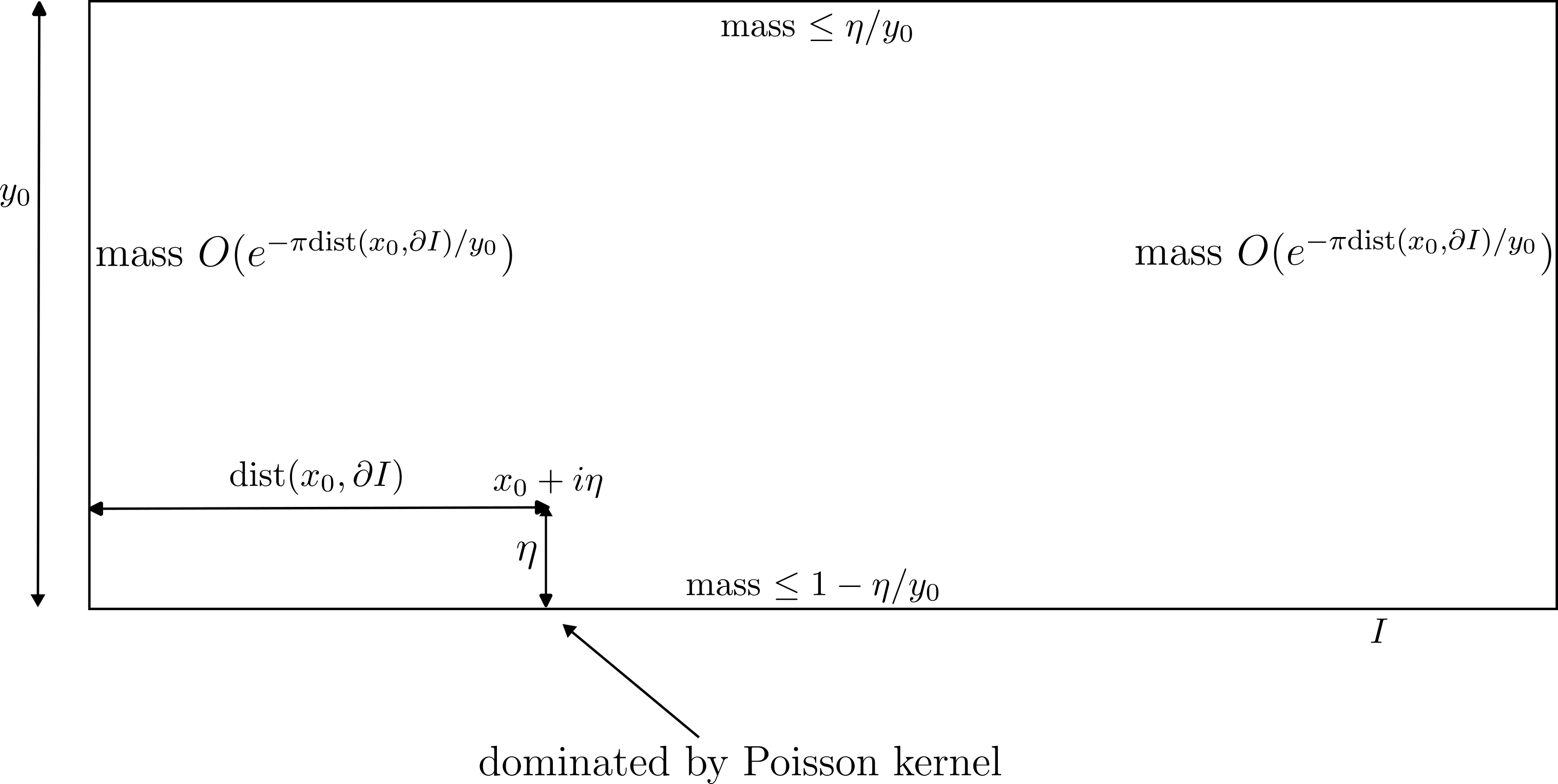}
    \caption{A schematic depiction of the distribution of harmonic measure on $\partial R^+(I,y_0)$ for a Brownian motion starting at $x_0+i\eta$.  }
    \label{fig-rect}
\end{figure}

\begin{proof}  For (i), this is clear since any Brownian motion starting at $x_0+i\eta$ which exits $R^+(I,y_0)$ through the lower edge, also exits the upper half-plane through the same location.

For the first part of (ii), apply \eqref{harmonic} to the harmonic function $1 - y / y_0$, and for the second part of (ii), apply \eqref{harmonic} to the harmonic function $y / y_0$.

For (iii), we rescale $y_0 = 1$ and $I = [-T,T]$.  We may assume that $\mathrm{dist}(x_0,\partial I) > 1$, as the claim is trivial otherwise. Using \Cref{lem-square}, we can check that the for any natural number $M$, the function
$$ u_M(x+iy) = \sum_{m \text{ odd}} \left(1-\frac{m}{M}\right)_+ \frac{4\sin(m\pi y)}{m\pi}  \frac{\cosh(m\pi x)}{\cosh(m \pi T)}$$
is harmonic on $R^+([-T,T],1)$, equals $0$ on the upper and lower edges, and converges almost everywhere to $1$ on the left and right edges in the limit $M \to \infty$.  Applying \eqref{harmonic} to this function and applying dominated convergence, we see that the mass assigned to the left and right edges is precisely $\lim_{M \to \infty} u_M(x_0+i\eta)$.  From the triangle inequality and routine estimation we have
$$ \lim_{M \to \infty} u_M(x_0+i\eta) \ll \sum_{m \text{ odd}} \frac{1}{m} e^{-\pi m \mathrm{dist}(x_0, \partial I) }.$$
Under the hypothesis $\mathrm{dist}(x_0, \partial I) \geq 1$, the right-hand side is $O( e^{-\pi \mathrm{dist}(x_0, \partial I)})$, giving the claim.\end{proof}

As a consequence of this lemma and the Nevanlinna two-constant theorem (see, e.g., \cite{garnett}), we obtain a local version of the
Phragm\'en--Lindel\"of three-lines theorem:

\begin{lemma}[Local Phragm\'en--Lindel\"of theorem]\label{tpl}  Let $f$ be a holomorphic function on a rectangle $R^+(I,y_0)$. Suppose that $f$ obeys the bound $|f(z)| \leq A$ on the upper and lower sides of $R^+(I,y_0)$, and $|f(z)| \leq AM$ on the left and right sides of $R^+(I,y_0)$, for some $A > 0$ and $M \geq 1$.  Then for any $x+iy \in R^+(I,y_0)$, one has
$$ |f(x+iy)| \leq A \exp\left( O(e^{-\pi \mathrm{dist}(x, \partial I) / y_0} \log M) \right).$$
\end{lemma}

\begin{proof}\footnote{An initial version of this proof was provided by ChatGPT.}  The function $\log \frac{|f|}{A}$ is subharmonic on $R^+([-T,T],1)$, and bounded by $0$ on the top and bottom sides of $R^+([-T,T],1)$ and by $\log M$ on the left and right sides.  From \eqref{subharmonic} and \Cref{harm-lemma}(iii) we conclude that
$$\log \frac{|f|}{A}(x+iy) \ll e^{-\pi \mathrm{dist}(x, \partial I) / y_0} \log M,$$
giving the claim.
\end{proof}

\begin{remark} By setting $I = [-T,T]$ and then sending $T \to \infty$, one recovers the familiar global Phragm\'en--Lindel\"of three-lines theorem, which asserts that if a function is holomorphic on a strip $\{ z: 0 \leq \mathrm{Im} z \leq y_0\}$, bounded by $A$ in magnitude on the edges of the strip, and grows slower than $\exp( e^{(\pi-\eps) |z|/y_0})$ for some $\eps>0$ in this strip, then it is bounded by $A$ in the interior of the strip.
\end{remark}

\section{Proof of local Bernstein estimates}\label{boas-sec}

In this section we prove \Cref{main-2}. Our arguments broadly follow those in \cite{duffin}.

We can rescale so that $A=1$, $\lambda=1$, and the point $x$ to which \eqref{la} is being tested is equal to $0$.  Thus $I$ now contains $[-L,L]$, and we have
\begin{itemize}
\item (Lower edge) $f(x)$ is real with $|f(x)| \leq 1$ for all $x \in I$.
\item (Upper edge) $e^{-y_0} |f(x+iy_0)| \leq 1$ for all $x \in I$.
\item (Left and right edges) $e^{-y} |f(x+iy)| \leq \exp (e^{\frac{\pi L}{4y_0}})$ for all $x \in \partial I$ and $0 \leq y \leq y_0$.
\end{itemize}
As the claim (i) follows from (ii), and $\exp(O( e^{-\frac{\pi L}{4y_0}} )) = 1 + O( e^{-\frac{\pi L}{4y_0}})$, our goals can now be simplified to the following:

\begin{itemize}
    \item[(ii)]  One has
$$
 |f(0) + i f'(0)| \leq \exp\left(O( e^{-\frac{\pi L}{4y_0}} )\right) \left(1 + O\left(\frac{1}{\min(y_0,L)}\right)\right).
$$
    \item[(iii)]  One has
$$
 |f(iy)| \leq \exp\left(O( e^{-\frac{\pi L}{4y_0}} )\right) \cosh \left( \left(1 + O\left(\frac{1}{\min(y_0,L)}\right) \right) y \right).
$$
    \item[(iv)] If $f(0) \neq 0$, then the number of zeroes of $f$ in the disk $D(0,r)$ is at most $O( 1 + r + \log \frac{1}{|f(0)|})$ for any $0 < r < \min(y_0,L)/4$.
\end{itemize}

Applying \Cref{tpl} to the holomorphic function $f(z) e^{iz}$, we conclude that
$$
 e^{-y} |f(x + iy)| \leq \exp\left(O( e^{-\frac{\pi L}{4y_0}} )\right)
$$
for all $x+iy \in R^+([-L/2,L/2], y_0)$.  Dividing $f$ by $\exp\left(O( e^{-\frac{\pi L}{4y_0}} )\right)$, we may now assume that
\begin{equation}\label{fiy}
 |f(x + iy)| \leq e^y
\end{equation}
for all $x+iy \in R^+([-L/2,L/2], y_0)$.  In particular, if $r < \min(y_0,L)/4$ one has
$$ |f(x+iy)| \leq e^{2r}$$
for all $x+iy \in D(0,2r)$, and the claim (iv) then follows from Jensen's formula.  We may simplify the remaining two claims (ii), (iii) to
\begin{equation}\label{la0}
 |f(0) + i f'(0)| \leq 1 + O\left(\frac{1}{\min(y_0,L)}\right)
\end{equation}
and
\begin{equation}\label{la1}
     |f(iy)| \leq \cosh \left( \left( 1 + O\left(\frac{1}{\min(y_0,L)}\right) \right) y\right)
\end{equation}
for $0 \leq y \leq y_0$.

By the Schwartz reflection principle, we can now extend $f$ to the rectangle $R([-L/2,L/2], y_0)$
with the bounds
$$ |f(x \pm iy)| \leq e^{|y|}$$
in this rectangle. In particular $|f(0)| \leq 1$.

If $\min(y_0,L) \ll 1$, then from the Cauchy inequalities (applied to a disk of radius $\min(y_0,L)/2$ around the origin) one has $f'(0) \ll \frac{1}{\min(y_0,L)}$, and the claim \eqref{la0} follows from the triangle inequality, while the bound \eqref{la1} in this case follows from \eqref{fiy}.  Thus we may assume that $y_0, L$ are larger than any given absolute constant.

Let $\eps \coloneqq \frac{C}{\min(y_0,L)}$ for a large absolute constant $C>0$.  We introduce the slightly rescaled, damped function
\begin{equation}\label{feps-def}
 f_{\eps}(z) \coloneqq f\left(\frac{z}{1+\eps}\right) \sinc\left(\frac{\eps}{1+\eps} z\right).
\end{equation}
This function is holomorphic on the rectangle $R([-L/2,L/2], (1+\eps) y_0)$
and we may bound
\begin{align*}
f_{\eps}(x+iy) &\ll \frac{e^{\frac{|y|}{1+\eps}} e^{\frac{\eps}{1+\eps} |y|}}{\eps (|x| + |y|)} \\
&= \frac{e^{|y|} \min(y_0,L)}{C(|x| + |y|)} \\
&\leq 2 \frac{\min(y_0,L)}{C(|x| + |y|)} \cosh y
\end{align*}
on this rectangle.  In particular, if $C$ is large enough, we have
\begin{equation}\label{f14}
 |f_{\eps}(x+iy)| < \frac{1}{4} \cosh y
\end{equation}
whenever $x+iy$ lies in $R([-L/2,L/2], (1+\eps) y_0)$ and is within distance $20\pi$ (say) to the boundary. Furthermore, for $-L/2 \leq x \leq L/2$, we may use the bound $|\sin(t)| \leq t$ to bound
\begin{equation}\label{feps}
 |f_{\eps}(x)| \leq 1
\end{equation}

Let $k$ be the largest integer such that $(k+10) \pi \leq L/2$, then $k$ can be assumed to be larger than any given absolute constant.  On the boundary of the rectangle
$$ R\left([\theta-k\pi, \theta+k\pi], (1+\eps) y_0\right)$$
the quantity $|\cos(z-\theta)|$ is equal to $\cosh(\mathrm{Im} y)$ on the vertical sides, and at least $\sinh(\mathrm{Im} y_0)$ on the horizontal siodes.  In either case we see from \eqref{f14} and the large nature of $y_0$ that
$$ |f_{\eps}(z)| < \frac{1}{2} |\cos(z - \theta)|.$$
Since $\cos(z-\theta)$ has exactly $2k$ zeroes in this rectangle, we conclude from Rouche's theorem that $\cos(z-\theta) - \alpha f_\eps(z)$ also has exactly $2k$ zeroes in this rectangle for any $0 <\alpha<1$.  On the other hand, on each of the $2k$ intervals $[\theta + j\pi, \theta + (j+1) \pi]$ for $-k \leq j < k$, we see from \eqref{feps} that $\cos(z-\theta) - \alpha f_\eps(z)$ is real-valued and changes sign; thus, by the intermediate value theorem, each such interval contains a zero in its interior.  We conclude that these are the complete set of zeroes, and that all the zeroes are simple.  In particular, $\cos(z-\theta) - \alpha f_\eps(z)$ cannot have a double zero at the origin for any $\theta \in [0,2\pi]$, or equivalently the quantity $f_\eps(0) + i f'_\eps(0)$ avoids the exterior region
$$\left\{ \frac{\cos \theta + i \sin \theta}{\alpha} : 0 \leq \theta \leq 2\pi; 0 < \alpha < 1\right\}$$
and thus
$$ |f_\eps(0) + i f'_\eps(0)| \leq 1.$$
From \eqref{feps-def} we thus have
$$ \left|f(0) + i \frac{f'(0)}{1+\eps}\right| \leq 1,$$
which places $f(0)+if'(0)$ in an ellipse around the origin of semi-major axis $1+\eps$ and semi-minor axis $1$, giving \eqref{la0}. This also gives the $y=0$ case of \eqref{la1}. For $y \neq 0$, we know that $\cos((1+\eps)iy-\theta) - \alpha f_\eps((1+\eps)iy)$ cannot vanish for any $0 < \alpha < 1$ and $\theta \in \R$.  From the trigonometric identity
$$ \cos((1+\eps)iy-\theta) = \cosh((1+\eps) y) \cos \theta + i \sinh((1+\eps) y) \sin \theta$$
we see that the curve
$$ \{ \cos((1+\eps)iy-\theta): 0 \leq \theta \leq 2\pi \}$$
is an ellipse with semi-major axis $\cosh((1+\eps) y)$ and semi-minor axis $\sinh((1+\eps) y)$ centered at the origin; see \Cref{fig-ellipse}.  From the previous discussion, we know that $\alpha f_\eps((1+\eps)iy)$ cannot lie on this ellipse for any $0 < \alpha < 1$, hence $f_\eps((1+\eps)iy)$ must lie on the boundary or interior of this ellipse.  In particular, $|f_\eps((1+\eps)iy)| \leq \cosh((1+\eps) y)$.
Since $\frac{\sinh \eps y}{\eps y} \geq 1$, we see from \eqref{feps-def} that $|f(iy)| \leq \cosh((1+\eps) y)$, as required.

\begin{figure}
    \centering
\includegraphics[width=0.5\textwidth]{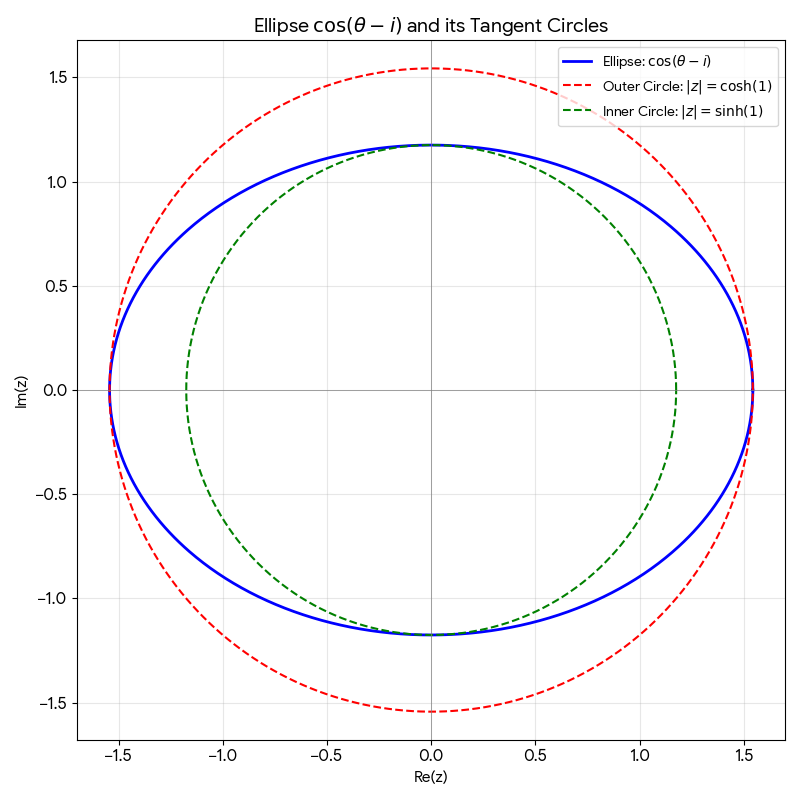}
    \caption{The ellipse $\{ \cos(ir-\theta): 0 \leq \theta \leq 2\pi \}$ lies between the disks $D(0,\sinh r)$ and $D(0,\cosh r)$; we illustrate this here with $r=1$.  (Image  generated by Gemini.)}
    \label{fig-ellipse}
\end{figure}

\begin{remark} The above argument in fact gives slightly sharper control on $f(x+iy)$, placing it in a certain ellipse; cf. the discussion in \cite[p. 240]{duffin}, as well as \cite{bern3}.
\end{remark}

\begin{remark} An anonymous commenter on my blog has provided an alternate proof of \Cref{main-2}, following similar lines to the proof of \Cref{tpl}, which we sketch as follows.  Again we normalize $A=\lambda=1$ and $x=0$. Let $\rho(w) = \log |w + \sqrt{w^2-1}| = \Re \operatorname{arccosh} w$ be the Green's function of $\C \backslash [-1,1]$ with pole at infinity, then the function $\rho(f(x+iy)) - y$ is subharmonic on the rectangle $R^+(I, y_0)$, vanishing on the lower edge, essentially nonpositive on the upper edge, and upper bounded by $O(e^{\frac{\pi}{4y_0}})$ on the side edges.  Applying \eqref{subharmonic} and \Cref{harm-lemma}(iii), we can obtain good upper bounds on $\rho(f(x+iy))-y$, which by the Cauchy--Riemann equations for $\operatorname{arccosh} f$ can also be used to upper bound the derivative of $\rho(f(x))$.  The claims (i)-(iii) of \Cref{main-2} can then be recovered from the chain rule.
\end{remark}

\section{Preliminary bounds}

In preparation for proving both components of \Cref{main}, we first investigate what one can say about the polynomial $P$, the measure $\mu$, and the logarithmic potential $U_\mu$ assuming a weak upper bound on $\lambda(x)$ in the interval $I$.

\begin{theorem}[Preliminary bounds]\label{prelim}  Let $I$ be a fixed interval in $[-1,1]$. Let $n$ be sufficiently large, and let $x_1,\dots,x_n$ be distinct points in $[-1,1]$.
Assume we have the weak upper bound
\begin{equation}\label{lambda-bound}
\sup_{x \in I} \lambda(x) \ll n^{O(1)}
\end{equation}
We allow all implied constants to depend on $I$, and the implied constants in \eqref{lambda-bound}.
\begin{itemize}
    \item[(i)] (Near-constant log-potential on $I$) There exist a real number
\begin{equation}\label{alpha}
 \alpha = O(1)
\end{equation}
such that
\begin{equation}\label{deltax}
 \delta(x) e^{-n\alpha} \leq |P(x)| \ll n^{O(1)} e^{-n\alpha}
\end{equation}
or equivalently (by \eqref{umu-def})
\begin{equation}\label{umu-bound} \alpha - O\left( \frac{\log n}{n} \right) \leq U_\mu(x) \leq \alpha + O\left( \frac{\log \frac{1}{\delta(x)}}{n} \right)
\end{equation}
for all $x \in I$, where $\delta$ is the distance function
\begin{equation}\label{delta-def}
 \delta(z) \coloneqq \min_{1 \leq i \leq n} |z-x_i|.
\end{equation}
    \item[(ii)]  (Lipschitz continuity of density) There exists a continuous function $\rho \colon I'\to \R$ obeying the Lipschitz bounds
\begin{equation}\label{goa}
 \rho(x_0) \asymp_{I'} 1; \quad \rho(x) = \rho(x_0) + O_{I'}(|x-x_0|)
\end{equation}
for all $x, x_0 \in I' \Subset I$, such that one has the estimate
\begin{equation}\label{umud}
 U_\mu(x \pm i\eta) = \alpha - \pi \eta \rho(x) + O_{I'}\left( \frac{\log n}{n} + \eta^3 \right)
\end{equation}
whenever $x \in I' \Subset I$ and $\eta \geq \frac{\log n}{n}$.  (See \Cref{fig-arcsine} and \Cref{fig-complex}.)
    \item[(iii)] (Density of state upper bound) If $I' \Subset I$, we have
\begin{equation}\label{mueta}
 \mu( [x-\eta, x+\eta] ) \ll_{I'} \eta + \frac{\log n}{n}
\end{equation}
when $x \in I'$ and $\eta > 0$.
    \item[(iv)] (Density of state asymptotic) If $I' \Subset I$, we have
\begin{equation}\label{muj}
 \mu(J) = \int_J \rho(x)\ dx + O_{I'}\left( \frac{|J|}{K} + \frac{K^2 \log n}{n} \right)
\end{equation}
for any subinterval $J$ of $I'$ and any $K \geq 1$.  (See \Cref{fig-arcsine}.)
\end{itemize}
\end{theorem}

\begin{example} In the case of Chebyshev nodes (\Cref{cheby}) one can take $\alpha=\log 2 + o(1)$, and $\rho$ to be the arcsine distribution $\rho(x) = \rho_{{\mathrm{as}}}(x) = \frac{1}{\pi \sqrt{1-x^2}}$ (up to $o(1)$ errors).  See \Cref{fig-arcsine} and \Cref{fig-pot}.
\end{example}

\begin{figure}
    \centering
\includegraphics[width=0.75\textwidth]{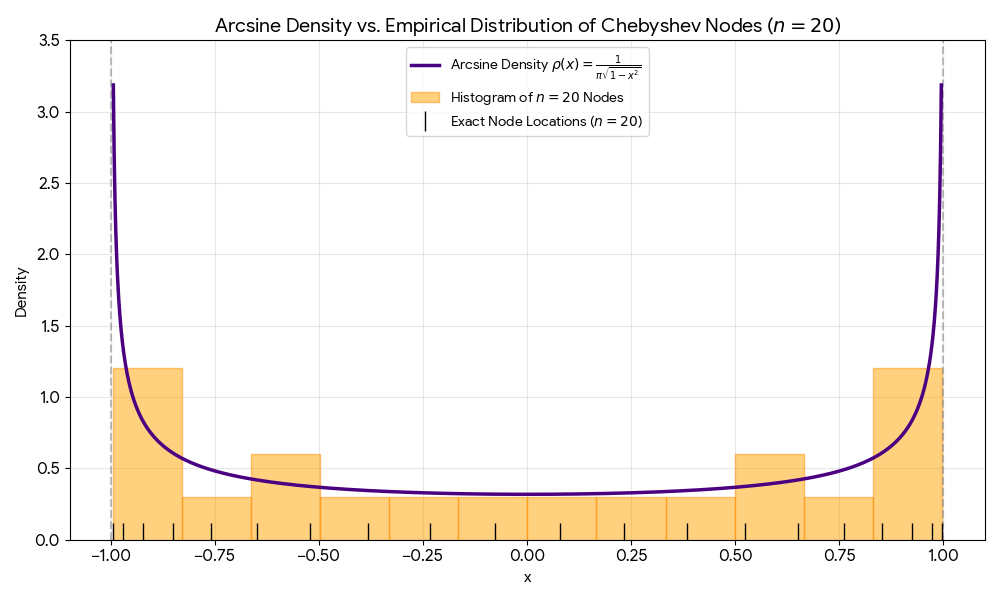}
    \caption{The arcsine distribution $\rho = \rho_{\mathrm{as}}$, which is Lipschitz continuous and comparable to $1$ in the bulk of $[-1,1]$, but develops singularities at the endpoints. This is superimposed with a normalized histogram of the Chebyshev nodes in \Cref{fig-cheby}.  (Image generated by Gemini.)}
    \label{fig-arcsine}
\end{figure}

\begin{figure}
    \centering
\includegraphics[width=0.75\textwidth]{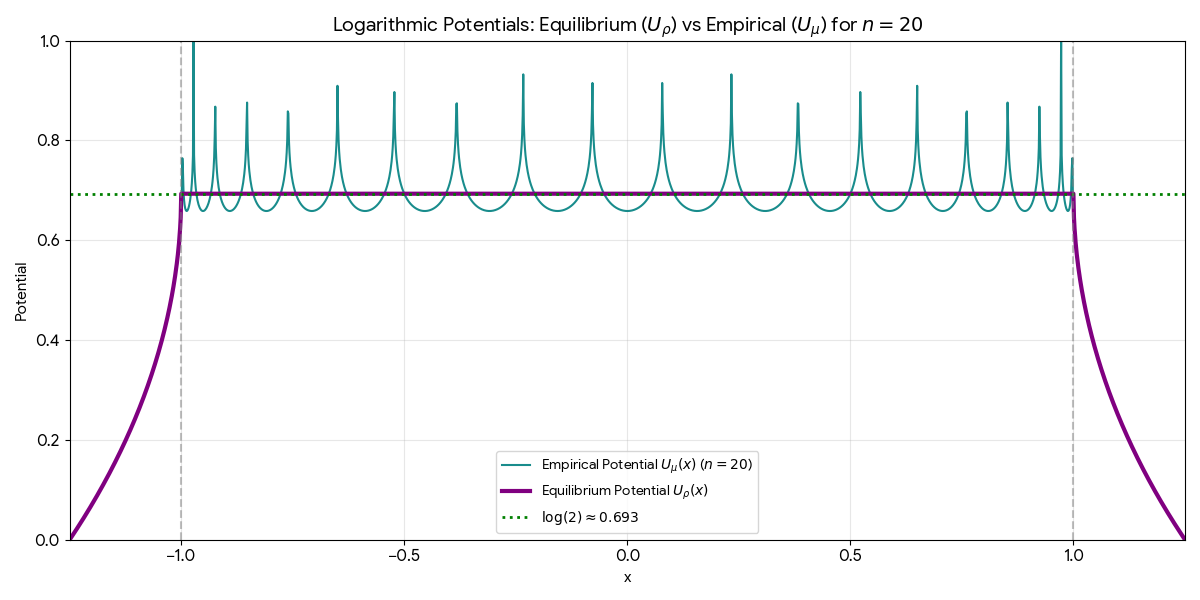}
    \caption{The logarithmic potential $U_\mu(x)$ for the polynomial in \Cref{fig-cheby}, compared against the potential $U_\rho(x) = \int_\R \log \frac{1}{|x-y|} \rho(y)\ dy$ for the arcsine distribution $\rho = \rho_{\mathrm{as}}$, which in this case is identically equal to $\alpha = \log 2$ on $[-1,1]$ (and equal to $\alpha - \log(|x| + \sqrt{x^2-1})$ outside of this interval).  The potential $U_\mu$ diverges logarithmically as the distance $\delta(x)$ to the nearest root of $P$ shrinks to zero.  (Image generated by Gemini.)}
    \label{fig-pot}
\end{figure}

\begin{figure}
    \centering
\includegraphics[width=0.75\textwidth]{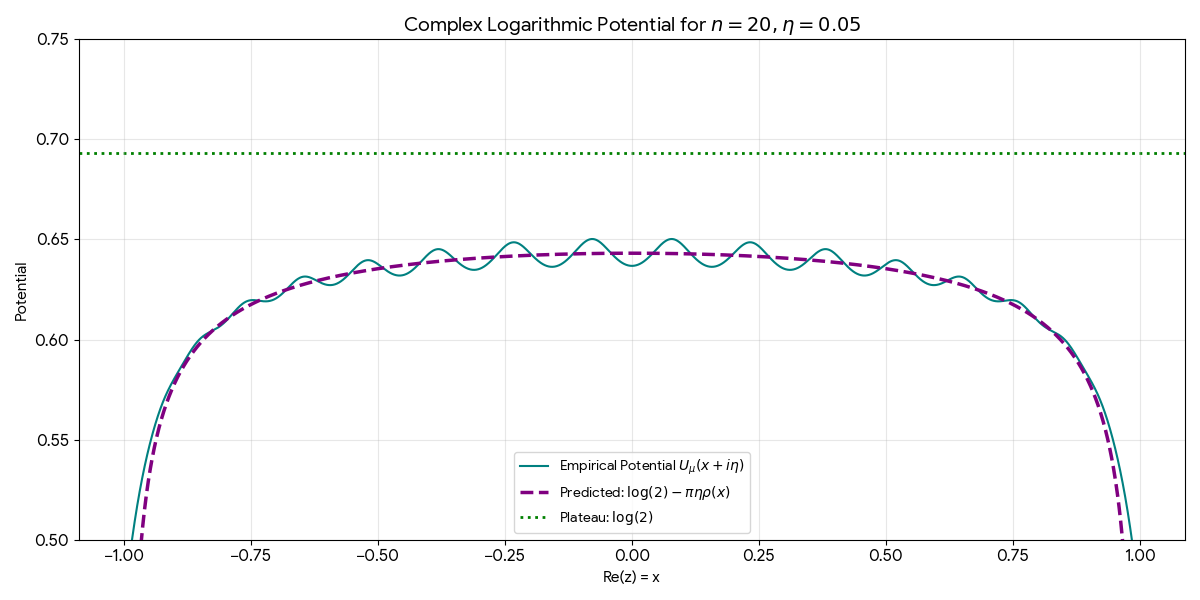}
    \caption{The logarithmic potential $U_\mu(x+i\eta)$ for the polynomial in \Cref{fig-cheby} and $\eta = 1/n$, compared with the predicted value of $\alpha - \pi \eta \rho(x)$ for $x$ in the bulk of $[-1,1]$ and $\rho = \rho_{\mathrm{as}}$.   The non-zero value of $\eta$ acts to damp the oscillations in the logarithmic potential caused by the discrete nature of the nodes; in particular, the potential is now smooth (and in fact harmonic in $x+i\eta$).  In practice, these oscillations become negligible once $\eta$ is significantly larger than $\log n/n$. (Image generated by Gemini.)}
    \label{fig-complex}
\end{figure}

\begin{proof} We first prove (i), following an argument provided to us by Sothanaphan \cite{gpt} using GPT.  Let $x \in I$.
Applying \eqref{interp} with $Q=1$ and then using the triangle inequality, we have the lower bound
\begin{equation}\label{lambda-lower}
\lambda(x) \geq 1,
\end{equation}
hence by \eqref{lambda-bound} we have $\lambda(x) = n^{O(1)}$.
If we define
\begin{equation}\label{alpha-def}
 \alpha \coloneqq \frac{1}{n} \log \left(\sum_{k=1}^n \frac{1}{|P'(x_k)|}\right),
\end{equation}
then from \eqref{lamax}, \eqref{delta-def} we have
$$ \frac{|P(x)|}{2} e^{n\alpha} \leq \lambda(x) \leq \frac{|P(x)|}{\delta(x)} e^{n\alpha}$$
giving \eqref{deltax}, and hence \eqref{umu-bound}.  From the local square-integrability of the logarithm $\log \frac{1}{|x|}$, \eqref{umu-def} and the triangle inequality, we have the bound
\begin{equation}\label{umisc}
 \int_{[-2,2]} U_\mu(x)^2\ dx \ll 1
\end{equation}
and hence by Cauchy-Schwarz
\begin{equation}\label{ul1}
\int_{[-2,2]} U_\mu(x)\ dx \ll 1.
\end{equation}
In particular, from Markov's inequality for integrals we have $U_\mu(x) = O(1)$ for all $x$ in a subset of $I$ of measure $\asymp 1$.  By \eqref{umu-bound}, one has $U_\mu(x) = \alpha + O(1)$ for at least one such $x$, giving \eqref{alpha} as required. This completes the proof of (i).

Now we prove (ii). In the region $\eta \geq 1$, direct calculation shows that
$$ U_\mu(x \pm i\eta) \ll_{I'} \log ( 1 + \eta )$$
for $x \in I'$, and the claim \eqref{umud} then follows from \eqref{alpha} (and \eqref{goa}).  Hence we may restrict to the case when $\frac{\log n}{n} \leq \eta < 1$.

As the function $U_\mu$ is harmonic outside of $x_1,\dots,x_n$ and has logarithmic growth at infinity, we may apply \eqref{poisson} (using standard limiting arguments to deal with the (mild) logarithmic singularities at $x_1,\dots,x_n$) to obtain
$$ U_\mu(x \pm i\eta) = \int_\R \Poisson_\eta(x-v) U_\mu(v)\ dv$$
for any $x \in \R$ and $\eta > 0$.  As the Poisson kernel $\Poisson_\eta$ has unit mass, we can rewrite this as
\begin{equation}\label{umo}
 U_\mu(x \pm i\eta) = \alpha + \int_\R \Poisson_\eta(x-v) (U_\mu(v) - \alpha)\ dv.
\end{equation}
We can estimate the portion of the integral in the region
$$ v \in I; \quad \delta(v) \geq n^{-10}$$
using \eqref{umu-bound}, and in the region
$$ v \in I; \quad \delta(v) < n^{-10}$$
using \eqref{umisc} and Cauchy--Schwarz, to conclude that
$$ U_\mu(x + i\eta) = \alpha - \pi \eta \rho( x + i \eta ) + O_{I'}\left( \frac{\log n}{n} \right)$$
in the region
$$ x \in I'; \quad \frac{\log n}{n} \leq \eta < 1,$$
where we define the (approximate) density of states $\rho$ by the formula
\begin{equation}\label{rho-def}
 \rho(x + i \eta) \coloneqq -\frac{1}{\pi^2} \int_{\R \backslash I} \frac{U_\mu(v) - \alpha}{(x-v)^2 + \eta^2} \ dv.
\end{equation}
From the separation between $I'$ and $\R \backslash I$, together with the bounds \eqref{alpha} and the easy bound
\begin{equation}\label{umu-large}
    U_\mu(z) \ll \log |z|
\end{equation}
whenever $|z| \geq 2$, we obtain the upper bound portions of \eqref{goa} as well as the more general Lipschitz bound
\begin{equation}\label{G-lip}
    \rho(x + i \eta) = \rho(x_0) + O_{I'}(\eta^2 + |x-x_0|)
\end{equation}
for any reference point $x_0 \in I'$.  The claim \eqref{umud} then follows. It remains to show that $\rho(x_0) \gg_{I'} 1$.  Let $\eta > 2\frac{\log n}{n}$ be chosen later.  From \eqref{umud} and the triangle inequality one has
$$ U_\mu\left(x_0 + i\frac{\log n}{n}\right) - U_\mu(x_0+i\eta) = \left(\eta - \frac{\log n}{n}\right) \rho(x_0) + O_{I'}\left( \frac{\log n}{n} + \eta^3 \right).$$
By \eqref{umu-def}, we have
$$ U_\mu\left(x_0 + i\frac{\log n}{n}\right) - U_\mu(x_0+i\eta) = \frac{1}{n} \sum_k \log \frac{(x_0-x_k)^2 + \eta^2}{(x_0-x_k)^2 + \frac{\log^2 n}{n^2}}.$$
Since $x_0 - x_k = O(1)$ and $\eta > 2 \frac{\log n}{n}$, we have
$$ \log \frac{(x_0-x_k)^2 + \eta^2}{(x_0-x_k)^2 + \frac{\log^2 n}{n^2}} \gg \eta^2$$
for all $k$, and hence (since $\eta - \frac{\log n}{n} \asymp \eta$)
$$ \rho(x_0) \gg \eta - O_{I'}\left( \frac{\log n}{n\eta} + \eta^2 \right).$$
Choosing $\eta \asymp_{I'} 1$ appropriately, we obtain the claim.

Now we prove (iii).  The function $U_\mu(x+i\eta)$ is harmonic in the upper half-plane, as is $\alpha - \pi \eta \rho(x_0)$ for any fixed $x_0$.  Applying elliptic regularity estimates on a disk of radius $c \eta |I|$ centered at $x+i\eta$ (as well as \eqref{goa}) for a sufficiently small $c>0$, we conclude that
\begin{equation}\label{umud-deriv}
    -\partial_\eta U_\mu(x+i\eta) = \pi \rho(x) + O_{I'}\left( \frac{\log n}{n \eta} + \eta \right)
\end{equation}
whenever
\begin{equation}\label{umud-deriv-region}
    x \in I'; \quad 2 \frac{\log n}{n} \leq \eta \leq \frac{1}{2}.
\end{equation}
Direct calculation using \eqref{umu-def} shows that $-\partial_\eta U_\mu$ is $\pi$ times the Poisson integral of $\mu$:
$$ -\partial_\eta U_\mu(x+i\eta) = \pi \int_\R \Poisson_\eta(x-u)\, d\mu(u).$$
Hence by \eqref{umud-deriv} we have
\begin{equation}\label{poisson-mu}
 \int_\R \Poisson_\eta(x-u)\, d\mu(u) = \rho(x) + O_{I'}\left( \frac{\log n}{n \eta} + \eta \right)
\end{equation}
Using \eqref{goa}, we conclude \eqref{mueta} in the case $2 \frac{\log n}{n} \leq \eta \leq 1/2$, which also implies the $0 < \eta < 2 \frac{\log n}{n}$ case by monotonicity. Finally, the case $\eta > 1/2$ is immediate as $\mu$ is a probability measure.

Now we prove (iv).  We may assume that $K \leq \sqrt{n/\log n}$, as the claim trivially follows from \eqref{goa} and the probability measure nature of $\mu$ otherwise. By \eqref{mueta}, dyadic decomposition, and decay of the Poisson kernel we have
$$ -\partial_\eta U_\mu(x + i\eta) = \pi \int_{x-K\eta}^{x+K\eta} \Poisson_\eta(x-v)\,d\mu(v) + O_{I'}\left(\frac{1}{K} \right)$$
for any $x \in I'$, so that
$$ \int_{x-K\eta}^{x+K\eta} \Poisson_\eta(x-v)\,d\mu(v) = \rho(x) + O_{I'}\left( \frac{\log n}{n \eta} + \eta + \frac{1}{K} \right).$$
Integrating this for $x \in J$ and using \eqref{mueta} to control errors, we conclude that
$$ \left( 1 - O\left(\frac{1}{K}\right)\right) \mu(J) + O(K\eta)  = \int_J \rho(x)\ dx + O_{I'}\left( \frac{\log n}{n \eta} |J| + \eta |J| + \frac{|J|}{K} \right)$$
for any $K \geq 1$ and $\frac{2 \log n}{n} \leq \eta \leq \frac{1}{2}$; setting $\eta \coloneqq \frac{2 K \log n}{n}$, we conclude\footnote{With more care one could improve the $K^2$ factor in \eqref{muj} further, but we will not need such an improvement for our purposes.} \eqref{muj}.
\end{proof}

\section{The amplitude function}\label{grid-sec}

Let the notation and hypotheses be as in \Cref{prelim}.  The density of state function $\rho$ introduced in that theorem controls the ``frequency'' of $P$: informally, $P$ oscillates at frequency $n \pi \rho(x)$ near any given $x \in I$.  In this section we complement this control with control of the ``amplitude'' of $P$.

More precisely, for any $x \in I$, we define the amplitude $A(x)$ by the formula
\begin{equation}\label{eqx}
A(x) \coloneqq \sup_{x' \in I} |P(x')| e^{-n^{0.1}|x-x'|}.
\end{equation}

We record the following properties of $A$.

\begin{proposition}[Amplitude function]\label{good-grid}  Let the notation and hypotheses be as in \Cref{prelim}.
\begin{itemize}
\item[(i)] (Domination by amplitude) We have $|P(x)| \leq A(x)$ for all $x \in I$.
\item[(ii)] (Log-lipschitz property)  For all $x,y \in I$ we have
$$ e^{-n^{0.1}|x-y|} A(y) \leq A(x) \leq e^{n^{0.1}|x-y|}A(y).$$
\item[(iii)]  (Amplitude size)  For any $x \in I$, we have $A(x) = n^{O(1)} e^{-n\alpha}$.
\item[(iv)]  (Location of extremizer) If $x \in I$, then there exists $x' \in I$ with
\begin{equation}\label{xxn}
    |x-x'| \ll \frac{\log n}{n^{0.1}}
\end{equation}
 such that
\begin{equation}\label{axi}
A(x) = |P(x')| e^{-n^{0.1}|x-x'|} = A(x') e^{-n^{0.1}|x-x'|}.
\end{equation}
\item[(v)]  (Bounds in the upper half-plane) If $x \in I' \Subset I$ and $0 \leq \eta \leq n^{-0.4}$, then
\begin{equation}\label{bound-1}
 |P(x+i\eta)| \leq (1 + O_{I'}(n^{-0.1})) A(x) \exp( \pi n \eta \rho(x) (1 + O_{I'}(n^{-0.1}))).
\end{equation}
and under the additional hypothesis $\eta \geq \log n / n$ we have the variant bound
\begin{equation}\label{bound-2}
|P(x+i\eta)| = n^{O_{I'}(1)} A(x) \exp( \pi n \eta \rho(x) ).
\end{equation}
\item[(vi)]  (Derivative bound) If $x \in I' \Subset I$, then
\begin{equation}\label{pderiv}
 |P'(x)| \leq  (1 + O_{I'}(n^{-0.1})) \pi n A(x) \rho(x).
\end{equation}
\end{itemize}

\end{proposition}

\begin{proof}
The claims (i), (ii) are immediate from \eqref{eqx} and the triangle inequality.  For (iii), we see from \eqref{eqx} and the upper bound in \eqref{deltax} that $A(x) \ll n^{O(1)} e^{-n\alpha}$.  For the lower bound, observe from the pigeonhole principle that we can find $x' \in I$ with $x' = x + O(n^{-0.1})$ and $\delta(x') \gg n^{-1.1}$, and then from \eqref{eqx} and the lower bound in \eqref{deltax} we have $A(x) \gg n^{-O(1)} e^{-n\alpha}$.

For (iv), we let $x'$ attain the supremum in \eqref{eqx} (which is attainable since we are taking the supremum of a continuous function over a compact set).  Then we have \eqref{axi} by \eqref{eqx} and the triangle inequality.  From part (iii), \eqref{axi}, and the upper bound in \eqref{deltax} we have
$$ n^{-O(1)} e^{-n\alpha} \ll A(x) = A(x') e^{-n^{0.1} |x-x'|} \ll n^{O(1)} e^{-n\alpha}$$
giving \eqref{xxn}.

Now we prove (v), (vi).  We allow implied constants to depend on $I'$. We will apply Theorem \ref{main-2} on the rectangle $R^+([x-n^{-0.2}, x+n^{-0.2}], n^{-0.4})$.  For $x'$ in the lower edge of this rectangle, we have from (i), (ii) that
$$ |P(x')| \leq A(x) (1 + O(n^{-0.1})).$$
For $\frac{\log n}{n} \leq \eta \leq n^{-0.4}$ and $x' \in [x-n^{-0.2}, x+n^{-0.2}]$,
we have from (iii) and \Cref{prelim}(ii) that
$$ U_\mu(x' + i \eta) = \frac{1}{n} \log \frac{1}{A(x)} - \pi \eta \rho(x') + O\left(\frac{\log n}{n}\right),$$
which when combined with \eqref{umu-def} gives \eqref{bound-2}.  On the upper edge of the rectangle, we conclude from \eqref{umu-def}, \eqref{goa} that
$$ |P(x'+in^{-0.4})| \leq A(x) \exp( n^{-0.4} \pi n \rho(x) (1 + O(n^{-0.1}))).$$
For $x'+iy$ on the left and right edges of the rectangle, we use the crude bound
$$P(x'+iy) \ll \prod_{i=1}^n O(1) \ll \exp(O(n)) \ll A(x) \exp( O(n ) )$$
using \eqref{poly-def}, (iii) and \eqref{alpha}.

Applying \Cref{main-2}(i), (iii) with $A = A(x) (1 + O(n^{-0.1}))$, $L = n^{-0.2}$, and $\lambda = \pi n \rho(x) (1 + O(n^{-0.1}))$, and using the bound $\cosh(y) \leq \exp(y)$ for $y \geq 0$, we conclude \eqref{pderiv} and \eqref{bound-1}.
\end{proof}

\begin{remark}\label{rem-amp} In the case of the Chebyshev polynomial (\Cref{cheby}), the amplitude $A(x)$ is equal to the constant $2^{1-n}$.  More generally, we expect that the amplitude $A(x)$ could vary somewhat at macroscopic scales, but will remain close to constant at microscopic or mesoscopic scales, as quantified by \Cref{good-grid}(ii).
\end{remark}

\section{Localizing around a good point}

Using \eqref{muj} and \eqref{pderiv}, we are led to the heuristic lower bound
$$ \lambda(x) \geq \int_{y \in I: |x-y| \geq 1/n} \frac{|P(x)|}{n \pi \rho(y) A(y) |x-y|}\ n \rho(y) dy$$
for $x \in I$, so in particular if $|P(x)|=A(x)$ we expect to have
\begin{equation}\label{lambda-low}
    \lambda(x) \geq \frac{1}{\pi} \int_{y \in I: |x-y| \geq 1/n} \frac{A(x)}{A(y) |x-y|}\ dy.
\end{equation}
If the amplitude $A(x)$ was constant then this would already give the desired lower bound $\frac{2}{\pi} \log n - O(1)$.

To make this heuristic precise, we first locate a candidate point $x_*$ for which the macroscopic contribution to the predicted lower bound \eqref{lambda-low} for $\lambda(x_*)$ is expected to be large.

\begin{proposition}[Good point]\label{gp}  Let the notation and hypotheses be in \Cref{prelim}, and let $I' \Subset I$.  Suppose $n$ is sufficiently large depending on $I'$. Then there exists $x_* \in I'$ such that
\begin{equation}\label{pxs}
|P(x_*)| = A(x_*)
\end{equation}
and
\begin{equation}\label{good-est}
    \frac{1}{\pi} \int_{y \in I': |y-x_*| \geq n^{-0.1}} \frac{A(x_*)}{A(y) |x_*-y|}\ dy \geq \frac{2}{\pi} \log n^{0.1} - O_{I'}(1).
\end{equation}
\end{proposition}

\begin{proof} Select an interval $I'' \subset I'$, and allow all implied constants to depend on $I', I''$; we assume $n$ sufficiently large depending on $I,I',I''$.  Direct calculation shows that
$$ \int_{[y-1,y+1]}\min\left( \frac{1}{|x-y|}, n^{0.1}\right) \ dx = 2 \log n^{0.1} - O(1)$$
for all $y \in I''$, and
$$ \int_{I''} \int_{[y-1,y+1] \backslash I''} \frac{1}{|x-y|}\ dx dy \ll 1$$
and hence by the triangle inequality
$$ \int_{I''} \int_{I''} \min\left( \frac{1}{|x-y|}, n^{0.1}\right) \ dx dy = 2 |I''| \log n^{0.1} - O(1).$$
From the arithmetic mean--geometric mean inequality
$$ \frac{1}{2} \frac{A(x)}{A(y)} + \frac{1}{2} \frac{A(y)}{A(x)} \geq 1$$
(cf. \cite[(3)]{esz}) and symmetry we thus have
$$ \int_{I''} \int_{I''} \frac{A(x)}{A(y)} \max\left( \frac{1}{|x-y|}, n^{0.1}\right)\ dx dy \geq 2 |I''| \log n^{0.1} - O(1).$$
By the pigeonhole principle, we can find $x \in I''$ such that
$$ \int_{I''} \frac{A(x)}{A(y)} \min\left( \frac{1}{|x-y|}, n^{0.1}\right)\ dy \geq 2 \log n^{0.1} - O(1).$$
By \Cref{good-grid}(iv), we can find $x_* = x +O(n^{-0.1} \log n)$ (so in particular $x_* \in I'$) such that $A(x_*) = |P(x_*)|$ and $A(x) = e^{-n^{0.1} |x_*-x|} A(x_*)$, thus
$$ \int_{I''} \frac{A(x_*)}{A(y)} e^{-n^{0.1} |x_*-x|} \min\left( \frac{1}{|x-y|}, n^{0.1}\right)\ dy \geq 2 \log n^{0.1} - O(1).$$
The function
$$ x \mapsto \min\left( \log \frac{1}{|x-y|}, \log n^{0.1} \right)$$
has a Lipschitz constant of $n^{0.1}$, hence
$$ e^{-n^{0.1} |x_*-x|} \min\left( \frac{1}{|x-y|}, n^{0.1}\right)
\leq \min\left( \frac{1}{|x_*-y|}, n^{0.1}\right).$$
We conclude that
$$ \int_{I''} \frac{A(x_*)}{A(y)} \min\left( \frac{1}{|x_*-y|}, n^{0.1}\right)\ dy \geq 2 \log n^{0.1} + O(1).$$
From \Cref{good-grid}(ii) we have
$$ \int_{|x_*-y| \leq n^{-0.1}} \frac{A(x_*)}{A(y)} n^{0.1}\ dy = O(1)$$
and thus
$$ \int_{I'': |x_*-y| > n^{-0.1}} \frac{A(x_*)}{A(y)} \frac{1}{|x_*-y|}\ dy \geq 2 \log n^{0.1} - O(1).$$
The claim follows.
\end{proof}

Pick an interval $I' \Subset I$, and let $x_* \in I'$ be the point in \Cref{gp}.  We allow implied constants to depend on $I'$, and assume $n$ sufficiently large depending on $I'$. To prove \Cref{main}(i), it will suffice by \eqref{lamax} to show that
\begin{equation}\label{total}
 \sum_{k=1}^n \frac{|P(x_*)|}{|P'(x_k)| |x_*-x_k|} \geq \frac{2}{\pi} \log n - O(1).
\end{equation}

Using \eqref{good-est}, we may readily control the macroscopic contribution:

\begin{lemma}[Macroscopic contribution]\label{macro}  We have
$$
 \sum_{k: |x_*-x_k| \geq n^{-0.1}} \frac{|P(x_*)|}{|P'(x_k)| |x_*-x_k|} \geq \frac{2}{\pi} \log n^{0.1} - O(1).
$$
\end{lemma}

\begin{proof}
By \eqref{pxs} and \Cref{good-grid}(iv) it suffices to show that
$$ \sum_{x_k \in I: |x_* - x_k| \geq n^{-0.1}} \frac{A(x_*)}{n \pi \rho(x_k) A(x_k) |x_*-x_k|} \geq
    \frac{1-O(n^{-0.1})}{\pi} \int_{y \in I': |y-x_*| \geq n^{-0.1}} \frac{A(x_*)}{A(y) |x_*-y|}\ dy - O(1).$$
We will just show the estimate
$$ \sum_{x_k \in I: x_k - x_* \geq n^{-0.1}} \frac{A(x_*)}{n \pi \rho(x_k) A(x_k) |x_*-x_k|} \geq  \frac{1-O(n^{-0.1})}{\pi} \int_{y \in I': y-x_* \geq n^{-0.1}} \frac{A(x_*)}{A(y) |x_*-y|}\ dy - O(1)$$
controlling the region to the right of $x_*$; the analogous estimate controlling the region to the left of $x_*$ is proven similarly.

If the region $\{ y \in I': y - x_* \geq n^{-0.1}\}$ has diameter less than $n^{-0.1}$, then from \Cref{good-grid}(ii) the integral is bounded and the claim is trivial.  Thus we may assume that this region has diameter at least $n^{-0.1}$, and so can be partitioned into intervals $J$ of length $\asymp n^{-0.2}$.  On each such interval, it will suffice to show that
$$ \sum_{x_k \in J: x_k - x_* \geq n^{-0.1}} \frac{A(x_*)}{n \pi \rho(x_k) A(x_k) |x_*-x_k|} \geq
    \frac{1-O(n^{-0.1})}{\pi} \int_J \frac{A(x_*)}{A(y) |x_*-y|}\ dy.$$
Let $x_J$ be the midpoint of $J$.  From \Cref{good-grid}(ii), \eqref{goa}, and elementary estimates we have
\begin{align*}
A(y) &= (1+O(n^{-0.1})) A(x_J) \\
|x_*-y| &= (1+O(n^{-0.1})) |x_*-x_J| \\
\rho(y) &= (1+O(n^{-0.1})) \rho(x_J)
\end{align*}
for all $y \in J$.  From this and \eqref{muj} (with $K = n^{0.1}$), we have $(1+O(n^{-0.1})) n \pi \rho(x_J) |J|$ nodes $x_k$ in $J$.  The claim follows.
\end{proof}

To show \eqref{total}, it now suffices to show that
\begin{equation}\label{total-micro}
 \sum_{|x_k - x_*| < n^{-0.1}} \frac{|P(x_*)|}{|P'(x_k)| |x_*-x_k|} \geq \frac{2}{\pi} \log n^{0.9} - O(1).
\end{equation}
It will be convenient to introduce the rescaled polynomial
$$ Q(z) \coloneqq \frac{P(x_* + \frac{z}{n \rho(x_*)})}{P(x_*)},$$
in which the amplitude and mean zero spacing have both been normalized to equal $1$.  The polynomial $Q$ has simple real roots
$$ v_k \coloneqq n \rho(x_*) (x_k - x_*).$$
The claim \eqref{total-micro} can then be rewritten (using \eqref{goa}) as
\begin{equation}\label{total-micro-rescaled}
     \sum_{|v_k| < T} \frac{1}{|Q'(v_k)| |v_k|} \geq \frac{2}{\pi} \log T - O(1)
\end{equation}
where
$$ T \coloneqq \rho(x_*) n^{0.9}.$$

Heuristically, $Q(z)$ should behave like the sinusoid $\cos(\pi z)$.  The following lemma helps make this intuition precise:

\begin{lemma}[$Q$ behaves like cosine]\label{Q-bound}\
\begin{itemize}
\item[(i)] (Normalization) We have $Q(0)=1$.
\item[(ii)] (Mesoscopic control of zeroes) For $n^{0.01} \leq t \leq T$, we have
$$ \# \{ k: |v_k| \leq t \} = (2 + O(n^{-0.001})) t.$$
\item[(iii)]  (Derivative bound)  For all $x \in [-T,T]$, one has
$$ |Q'(x)| \leq (1 + O(n^{-0.1})) \pi \exp( O( |x|/T )).$$
\item[(iv)]  (Microscopic control of zeroes) For any $0 < t \leq T$, one has
\begin{equation}\label{fz}
    \# \{ k: |v_k| \leq t \} \ll 1 + t.
\end{equation}
\item[(v)] (Local exponential type bound) For $x, y = O(n^{0.1})$, one has
\begin{equation}\label{expt}
Q(x+iy) \ll \exp( \pi |y| )
\end{equation}
for $x, y = O(n^{0.1})$, with the variant bound
\begin{equation}\label{exp-var}
|Q(x+iy)| = n^{O(1)} \exp( \pi |y| )
\end{equation}
under the additional hypothesis $|y| \geq \rho(x_*) \log n$.
\end{itemize}
\end{lemma}

\begin{proof}  The claim (i) is clear from construction, and (ii) follows from \Cref{prelim}(iv) (with $K = n^{0.001}$), \eqref{goa}, and rescaling.

To prove (iii), we see from \Cref{good-grid}(vi) and rescaling that
$$ |Q'(x)| \leq \pi \frac{\rho(x')}{\rho(x_*)} A(x')
 (1 + O(n^{-0.1}))$$
where $x' \coloneqq x_* + \frac{x}{n \rho(x_*)}$.  The claim then follows from \eqref{goa} and \Cref{good-grid}(ii).

The bounds in (v) follow from \Cref{good-grid}(v) (and \eqref{goa}).  To prove (iv), it suffices by (ii) to assume $0 < t \leq n^{0.01}$, and then the claim follows from \Cref{main-2}(iv) applied to the rectangle $R^+([-n^{0.1}, n^{0.1}], n^{0.05})$ (say) with $L = n^{0.1}$, $A = 1 + O(n^{-0.1})$, $\lambda = 1 + O(n^{-0.1})$, after using (i) and (v).
\end{proof}

The remaining task in proving \Cref{main}(i) is to deduce \eqref{total-micro-rescaled} from \Cref{Q-bound}.  This will be the objective of the next section.

\section{Controlling the microscale contribution}\label{micro-sec}

Our arguments here are inspired from those in \cite{erdos}.

Call a scale $2 \leq t \leq T$ \emph{good} if
\begin{equation}\label{good} \# \{ k: |v_k| \leq t \} \geq \left( 2 - \frac{1}{\log^2 t} \right) t.
\end{equation}
By \Cref{Q-bound}(ii) we see that any scale $t$ in the mesoscopic range for all $n^{0.01} \leq t \leq T$ is good.  We make the following observation:

\begin{lemma}[Contribution of good scales]\label{good-cont}  Suppose that $2 \leq t_0 \leq T$ is such that all scales $t$ in the range $t_0 < t \leq T$ are good.  Then
$$     \sum_{t_0 \leq |v_k| < T} \frac{1}{|Q'(v_k)| |v_k|} \geq \frac{2}{\pi} \log \frac{T}{t_0} - O(1).$$
\end{lemma}

\begin{proof} By increasing $t_0$ if necessary we may assume that $t_0$ is larger than any given absolute constant. For $j=1,\dots,k$, let $\tilde v_j$ be the root of $Q$ with the $j^{\mathrm{th}}$ smallest magnitude.  Then by \eqref{good} we see that for any $4t_0 \leq j \leq T$ one has
$$ |\tilde v_j| \leq \left( 1 - O\left(\frac{1}{\log^2 j}\right) \right) \frac{j}{2}$$
(in particular $\tilde v_j$ lies outside of $[-t_0,t_0]$), while from \Cref{Q-bound}(ii) we also have $|\tilde v_j| \leq T$.  Applying \Cref{Q-bound}(iii), we conclude that
$$     \sum_{t_0 \leq |v_k| < T} \frac{1}{|Q'(v_k)| |v_k|}
\geq \sum_{4t_0 \leq j \leq T} \frac{1}{\left( 1 + O\left(\frac{1}{\log^2 j}\right) \right) \frac{j}{2} \exp( O( j/T )) }.$$
Writing
$$\frac{1}{\left( 1 + O\left(\frac{1}{\log^2 j}\right) \right) \frac{j}{2} \exp( O( j/T )) } = \frac{2}{j} - O\left( \frac{1}{j \log^2 j} + \frac{1}{T} \right)$$
and summing the harmonic series, we obtain the claim.
\end{proof}

Let $t_0$ be the least scale for which the hypothesis of \Cref{good-cont} holds, then by the previous discussion we have $2 \leq t_0 \leq n^{0.01}$.  If $t_0$ is bounded then \eqref{total-micro-rescaled} follows from \Cref{good-cont}, so we may assume that $t_0$ is larger than any given absolute constant.  Since the left-hand side of \eqref{good} is piecewise continuous with jumps of at most $2$, we have that the number
$$ \ell \coloneqq \# \{ k: |v_k| \leq t_0 \} $$
of zeroes in the microscale region $[-t_0,t_0]$ is given by
\begin{equation}\label{lcount} \ell = \left( 2 - \frac{1}{\log^2 t_0} \right) t_0 + O(1).
\end{equation}

We may assume the bound
\begin{equation}\label{t1}
     \sum_{|v_k| < t_1} \frac{1}{|Q'(v_k)| |v_k|} \leq \frac{2}{\pi} \log t_1
\end{equation}
for any $t_0 \leq t_1 \leq T$, since otherwise we can obtain \eqref{total-micro-rescaled} from \Cref{good-cont} (with $t_0$ replaced by $t_1$).

The bound \eqref{lcount} is in conflict with the previous intuition that $Q$ is supposed to behave like $\cos(\pi x)$, as it asserts that $Q$ has noticeably fewer zeroes than this sinusoid in the microscale region $[-t_0,t_0]$.  Following \cite{erdos}, our strategy is to exploit this conflict by studying a ``blending'' $G$ of $Q$ and a sinusoid, which will vanish or be small at all of the $v_j$, while still being large at the origin and of lower exponential type than $\cos(\pi x)$, which will turn out to contradict \eqref{t1} after a residue theorem calculation (analogous to the one in \Cref{res}).

We turn to the details. Let the $\ell$ zeroes of $Q$ in $[-t_0,t_0]$ be denoted by $w_1,\dots,w_\ell$.  The sinusoid
\begin{equation}\label{ka}
 \sin\left( \pi \left(1 - \frac{1}{\log^3 t_0}\right) z\right)
\end{equation}
has zeroes at the integer multiples of $(1 - \frac{1}{\log^3 t_0})^{-1}$.
From \eqref{lcount}, we can therefore locate $\ell$ zeroes $u_1,\dots,u_\ell$ of \eqref{ka} that lie in the interval
\begin{equation}\label{tlog}
 \left[-\left(t_0-\frac{t_0}{\log^3 t_0}\right), t_0 - \frac{t_0}{\log^3 t_0}\right].
\end{equation}
There is enough freedom in this selection that we can ensure the technical condition
\begin{equation}\label{tech}
 \sum_{j=1}^\ell u_j \neq \sum_{j=1}^\ell w_j.
\end{equation}
After removing singularities, the ``blended'' function
$$ G(z) \coloneqq \frac{\prod_{j=1}^\ell (z-w_j)}{\prod_{j=1}^\ell (z-u_j)} \sin\left( \pi \left(1 - \frac{1}{\log^3 t_0}\right) z\right) $$
is then in the Bernstein class ${\mathcal B}^{\infty,\R}_{\pi(1 - \frac{1}{\log^3 t_0})}$.  Informally, $G$ shares many zeroes in common with $Q$ in the interval $[-t_0,t_0]$, but behaves like the sinusoid \eqref{ka} at infinity.  More precisely, for $x \to \infty$, we have the asymptotic
$$ G(x) = \sin\left( \pi \left(1 - \frac{1}{\log^3 t_0}\right) x\right)
\left(1 + \frac{\sum_{j=1}^\ell u_j - \sum_{j=1}^\ell w_j}{x} + O\left(\frac{1}{x^2}\right)\right)$$
and so we see (using \eqref{tech}) that $G$ is not a sinusoid, and $|G(x)|$ attains its maximal value at some finite $x_{**}$ with
\begin{equation}\label{gx1}
|G(x_{**})| > 1.
\end{equation}
From \Cref{ds}(vi), we see that $x_{**}$ stays a distance strictly greater than $(1 - \frac{1}{\log^3 t_0})^{-1}/2$ from the zeroes of $G$.  As this set of zeroes contains all the integer multiples of $(1 - \frac{1}{\log^3 t_0})^{-1}$ outside of \eqref{tlog}, we conclude that $x_{**}$ lies in the $O(1)$-neighborhood of \eqref{tlog}, so in particular
\begin{equation}\label{xst}
 |x_{**}| \leq t_0 - \frac{t_0}{2\log^3 t_0}.
\end{equation}
We now localize $G$ further around $x_{**}$ (at the cost of increasing the exponential type slightly) by introducing the entire function
\begin{equation}\label{hforth}
 H(z) \coloneqq G(z) \sinc\left(\frac{z - x_{**}}{t_0^{1/3}}\right)^{\lfloor t_0^{1/6} \rfloor}.
\end{equation}
The point is that the exponential type of this function remains slightly less than $\pi$, but this function decays rapidly on the real line outside of the region $x = x_{**} + O(t_0^{1/3})$, so in particular becomes extremely small outside of $[-t_0,t_0]$.

Informally, $H$ is large at $x_{**}$ and vanishing or small at the zeroes of $Q$, while remaining ``lower degree'' than $Q$, which is in conflict with the Lebesgue interpolation formula \eqref{interp} when combined with the bounds in \Cref{good-cont}.  To make this intuition rigorous, we will rely upon the residue theorem, in the spirit of \Cref{res}.  We introduce the meromorphic function
$$\frac{H(z)}{(z-x_{**}) G(x_{**}) Q(z)}.$$
By construction, this function has a simple pole at $x_{**}$ with residue $H(x_{**})/G(x_{**}) Q(x_{**}) = 1 / Q(x_{**})$, removable singularities at $w_1,\dots,w_l$, and simple poles at the other zeroes $v_j$ of $Q$ outside of $[-t_0,t_0]$, with residue $H(v_j)/(v_j-x_{**}) G(x_{**}) Q'(v_j)$ at each such zero.  By \eqref{fz} and the pigeonhole principle, we can find $T' \asymp n^{0.1}$ such that $\pm T'$ stays at a distance $\gg 1$ from all the zeroes $v_j$ of $Q$.  From \eqref{deltax} and \Cref{good-grid}(c) we have the crude bound
\begin{equation}\label{fpt}
 |Q(\pm T')| = n^{O(1)} \frac{e^{-n\alpha} }{ |P(x_*)| }= n^{O(1)}.
\end{equation}

Inspired by \Cref{res}, we apply the residue theorem (\Cref{residue}) with $\gamma$ being the contour traversing the rectangle
$$ R([-T',T'], \log^4 t_0 \log n)$$
once anticlockwise, to conclude the identity
\begin{equation}\label{resident}
\frac{1}{2\pi i} \int_\gamma \frac{H(z)}{(z-x_{**}) G(x_{**}) Q(z)}\ dz =  \frac{1}{Q(x_{**})} + \sum_{j: t_0 < |v_j| \leq T'} \frac{H(v_j)}{(v_j - x_{**}) G(x_{**}) Q'(v_j)}.
\end{equation}

From \eqref{expt} we have a large residue at $x_{**}$:
\begin{equation}\label{ft}
 \left| \frac{1}{Q(x_{**})}\right| \gg 1.
\end{equation}

We can use \eqref{t1} to make the sum of the other residues much smaller:

\begin{lemma}[Other residues small]\label{ressum}  We have
\begin{equation}\label{reseq}
       \sum_{j: t_0 < |v_j| \leq T'} \frac{H(v_j)}{(v_j - x_{**}) G(x_{**}) Q'(v_j)}  \ll \frac{\log^5 t_0}{t_0^{10}}.
\end{equation}
\end{lemma}

\begin{proof}
Because $v_j$ lies outside $[-t_0,t_0]$, we see from \eqref{xst} that $|v_j - x_{**}| \gg |v_j|/\log^3 t_0$.  From this we can bound the left-hand side of \eqref{reseq} by
$$ \ll (\log^3 t_0) \sum_{j: t_0 < |v_j| \leq T'} \frac{|H(v_j)|}{|G(x_{**})| |v_j| |Q'(v_j)|} $$
By definition of $x_{**}$ we have $|G(v_j)| \leq |G(x_{**})|$, and hence by \eqref{hforth} and crude estimation
$$ |H(v_j)| \ll \frac{|G(x_{**})|}{|v_j|^{10}} $$
(say).  We can thus bound the preceding expression by
$$ \ll (\log^3 t_0) \sum_{j: t_0 < |v_j| \leq T'} \frac{1}{|v_j|^{11} |Q'(v_j)|}.$$
The claim now follows from \eqref{t1} and dyadic decomposition.
\end{proof}

Finally, using \eqref{exp-var}, we can make the integrand negligible:

\begin{lemma}[Integrand negligible]\label{integ-lemma}  For any $z$ in $\gamma$ one has
\begin{equation}\label{integrand-bound}
      \frac{|H(z)|}{|G(x_{**})| |z-x_{**}| |Q(z)|}
\ll n^{-10}
\end{equation}
\end{lemma}

\begin{proof}
First suppose that $z$ lies on the upper edge
$$ \{ x + i \log^4 t_0 \log n : -T' \leq x \leq T' \}.$$
From \eqref{exp-var} that
$$ \frac{1}{|z-x_{**}| |Q(z)|} \ll n^{O(1) - \pi \log^4 t_0}. $$
From \Cref{ds}(iii) we have
$$ |G(z)| \ll |G(x_{**})| n^{\pi \log^4 t_0 - \pi \log t_0}$$
and then by \eqref{hforth} and the crude bound
$$|\sinc(z)| = \left|\int_0^1 \cos(tz)\ dt\right| \leq \cosh(\mathrm{Im} z) \leq \exp(|\mathrm{Im} z|)$$
one has
$$ |H(z)| \ll |G(x_{**})| n^{\pi \log^4 t_0 - \pi \log t_0} \exp\left( \frac{\log^4 t_0 \log n}{t_0^{1/3}} \lfloor t_0^{1/6} \rfloor \right).$$
One can check that the final term can be absorbed in the $n^{-\pi \log t_0}$ term with room to spare, giving \eqref{integrand-bound} on the upper edge. The lower edge is treated similarly.

Now suppose $z$ lies on the right edge
$$ \{ T' + iy : -\log^4 t_0 \log n \leq y \leq \log^4 t_0 \log n \}.$$
From \eqref{fpt} that
$$ |Q(z)| \geq |Q(T')| = n^{O(1)}$$
while $|z-x_{**}| \asymp T'$.  From \Cref{ds}(iii) we have
$$ |G(z)| \ll |G(x_{**})|n^{O(\log^4 t_0)}$$
and then
$$      \frac{|H(z)|}{|G(x_{**})| |z-x_{**}| |Q(z)|}\ll \frac{n^{O(\log^4 t_0)} \exp\left( \frac{\log^4 t_0 \log n}{t_0^{1/3}} \lfloor t_0^{1/6} \rfloor \right)}{n^{O(1)} (T')^{\lfloor t_0^{1/6} \rfloor+1}}.$$
Since $T' \asymp n^{0.1}$, we readily conclude \eqref{integrand-bound}, as the $(T')^{\lfloor t_0^{1/6} \rfloor}$ denominator is more than strong enough to absorb various terms in the numerator while producing the desired gain $n^{-10}$.
\end{proof}

These two lemmas and \eqref{ft} contradict \eqref{resident} if $t_0$ (and $n$) are large enough.  This contradiction concludes the proof of \Cref{main}(i).

\section{Proof of \texorpdfstring{\Cref{main}}{Theorem~\ref{main}}(ii)}\label{integral-sec}

Now we prove \Cref{main}(ii).  Let $I_0 \subset [-1,1]$ be fixed.  We allow implied constants to depend on $I_0$.  Our task is to show that
\begin{equation}\label{io}
 \int_{I_0} \lambda(x)\ dx \geq \frac{(4-o(1))|I_0|}{\pi^2} \log n
\end{equation}

\subsection{A crude upper bound}

We first establish some crude bounds on $\lambda$ that admit polynomial losses of $n^{O(1)}$, in order to be able to access \Cref{prelim}.

Clearly we may assume without loss of generality that
$$ \int_{I_0} \lambda(x)\ dx \ll \log n.$$
By \eqref{lambda-def} this implies that
$$ \int_{I_0} |\ell_k(x)|\ dx \ll \log n$$
for each $k$.  As $\ell_k$ is a polynomial of degree $n-1$, the Nikolskii\footnote{One could also use Markov's inequality for polynomials and the fundamental theorem of calculus here.} inequality \cite{kamzolov} then gives the somewhat crude bound
$$ \sup_{x \in I_0} |\ell_k(x)| \ll n^2 \log n$$
and hence by \eqref{lambda-def} and the triangle inequality
$$ \sup_{x \in I_0} \lambda(x) \ll n^3 \log n.$$
Thus \Cref{prelim} is applicable (with $I=I_0$).  In particular, the density of state function $\rho(x)$ is now available on $I_0$.

\begin{remark}
In principle we could also apply \Cref{good-grid} to obtain an amplitude function $A(x)$, but as it turns out we will not need this function in this section; the amplitude function gives good control on $P(x_*)$ for some special points $x_*$, which is important for proving \Cref{main}(i), but is much less effective at controlling the average behavior of this polynomial, which is what is relevant for \Cref{main}(ii).  We will however introduce a quantity $P_0$ that plays a somewhat analogous role to $A(x)$ later in our arguments.
\end{remark}

\subsection{Localizing to a mesoscopic interval}

To prove \eqref{io}, it will suffice by a standard diagonalization argument to show that
$$ \int_{I_0} \lambda(x)\ dx \geq \frac{(4-o(1))|I_1|}{\pi^2} \log n$$
for any fixed $I_1 \Subset I_0$.  Henceforth we fix this choice of $I_1$, and allow implied constants to depend on $I_1$.  By \eqref{lamax}, we can write the left-hand side as
$$ \int_{I_0} \sum_{k} \frac{|P(x)|}{|x-x_k| |P'(x_k)|}\ dx.$$

It will be convenient to average away the weight $|x-x_k|$.
Let $\eps > 0$ be a sufficiently small fixed constant. Suppose we can show that
\begin{equation}\label{jo} \int_{J_0} \sum_{k: x_k \in J_0} \frac{|P(x)|}{|P'(x_k)|}\ dx \geq \frac{(2 - O(\eps))|J_0|^2}{\pi^2}\end{equation}
for any ``mesoscopic'' interval $J_0$ (which is not fixed with respect to $n$) with center in $I_1$ and length $n^\eps \leq |J_0| \leq n^{1-\eps}$, assuming $n$ sufficiently large depending on $\eps$.  Taking $J_0 = [y-r, y+r]$ (so that $2|J_0|^2 = 8r^2$), dividing both sides by $r^3$, and integrating in $r$ and $y$, we conclude that
$$ \int_{I_1} \int_{n^\eps/2}^{n^{1-\eps}/2} \int_{y-r}^{y+r} \sum_{k: x_k \in [y-r, y+r]} \frac{|P(x)|}{|P'(x_k)|}\ \frac{dx dr dy}{r^3}
 \geq \frac{(8 - O(\eps)) |I_1|}{\pi^2} \log n.$$
Note in the left-hand side that $x$ is restricted to $I_0$ (if $n$ is large enough), and one can upper bound this expression by
$$ \int_{I_0} \sum_k \frac{|P(x)|}{|P'(x_k)|} \left( \int_{|x-x_k|/2}^\infty \left( \int_{[x-r,x+r] \cap [x_k-r, x_k+r]} 1\ dy \right)\ \frac{dr}{r^3} \right)\ dx.$$
One can calculate that
$$ \int_{|x-x_k|/2}^\infty \left( \int_{[x-r,x+r] \cap [x_k-r, x_k+r]} 1\ dy \right)\ \frac{dr}{r^3} =
\int_{|x-x_k|/2}^\infty \frac{2r - |x-x_k|}{r^3}\ dr =
\frac{2}{|x-x_k|}.$$
We conclude that
$$ \int_{I_0} \sum_k \frac{|P(x)|}{|x-x_k| |P'(x_k)|}\ dx
\geq \frac{(4-O(\eps))|I_1|}{\pi^2} \log n$$
for $n$ sufficiently large depending on $\eps$, and the claim follows by letting $\eps \to 0$.

\subsection{Preliminary bounds}

Our remaining objective is to establish \eqref{jo}.  The first step is to adapt some arguments from \cite{esz} to obtain some reasonably good control (off from the truth by factors of $O(1)$) on $P$ in the interval $J_0$, which are more accurate in some respects from what one can obtain from \Cref{prelim}.

\begin{remark} The control on $P$ we will obtain here resembles somewhat the assertion that $|P(x)|$ is an $A_2$ weight in the sense of Muckenhoupt \cite{muckenhoupt}, although one should not take this analogy too literally since $1/|P(x)|$ is not quite locally integrable on the real axis (although we will later be able to obtain reasonable control on integrals of $1/|P(x+in^{-1})|$), and also we will encounter some exceptional sets where we lose some control on $P$.  In any event, we will not directly use the theory of such weights here, although experts familiar with that theory will note other similarities (for instance that $\log |P(x)|$ will behave somewhat like a function of bounded mean oscillation).
\end{remark}

We first recall a simple lemma from \cite[Lemma IV]{et3}.

\begin{lemma}[Lower bound on interpolation functions]\label{consec}  For any $1 \leq k < n$ and $x_k \leq x \leq x_{k+1}$, one has
$$ \ell_k(x) + \ell_{k+1}(x) \geq 1.$$
(See \Cref{fig-consec}.)
\end{lemma}

\begin{figure}
    \centering
\includegraphics[width=0.75\textwidth]{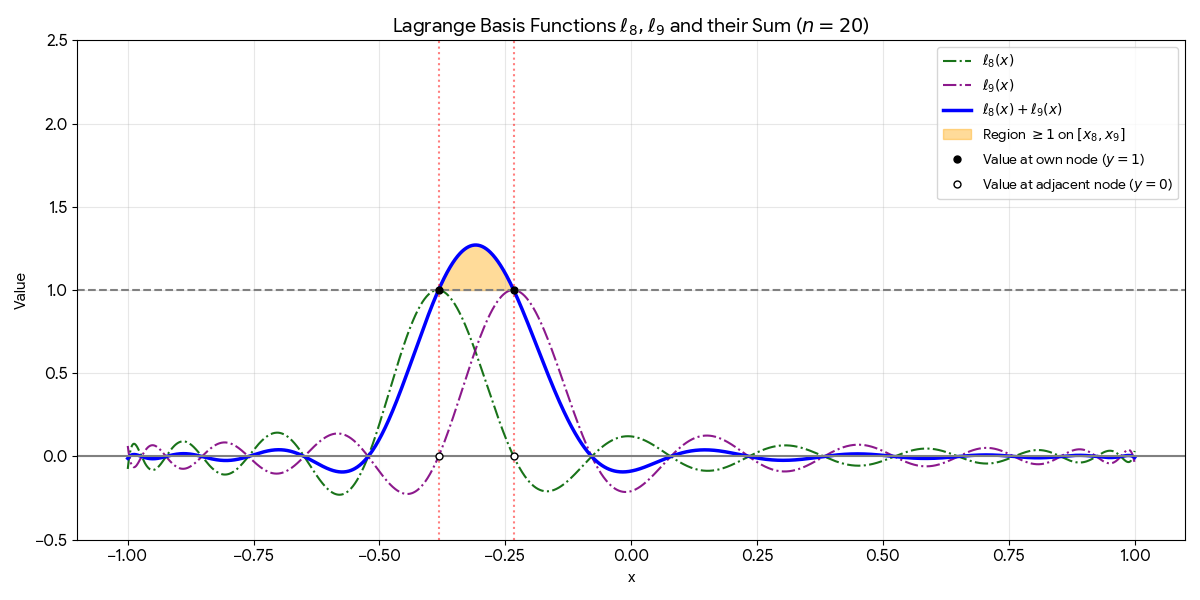}
    \caption{An illustration of \Cref{consec} with $k=8$, $n=20$, and $x_k$ the Chebyshev nodes in \Cref{cheby}. (Image generated by Gemini.)}
    \label{fig-consec}
\end{figure}

\begin{proof} For completeness we provide a proof here. By construction, $\ell_k+ \ell_{k+1}$ is a polynomial of degree at most $k-1$ that equals $1$ at $x_k, x_{k+1}$ and vanishes at all other $x_i$.  By the factor theorem, we thus have
$$\ell_k(x)+ \ell_{k+1}(x) = (ax+b) \prod_{i \neq k,k+1} |x-x_i|$$
for $x_k \leq x \leq x_{k+1}$, for some linear function $ax+b$ that is positive at $x_k$, $x_{k+1}$ and thus also for the entire interval $x_k \leq x \leq x_{k+1}$.  Every factor on the right is log-concave, thus $\ell_k + \ell_{k+1}$ is log-concave on this interval.  Since it equals $1$ at the endpoints $x_k$, $x_{k+1}$, the claim follows.
\end{proof}

\begin{corollary}\label{stosh}  For any $1 \leq k < n$, $x_k \leq x \leq x_{k+1}$, and $y \in \R$, one has
$$ |P(x+iy)| \geq \frac{\delta(x+iy)}{2} p_k$$
where
\begin{equation}\label{pk}
p_k \coloneqq \min( |P'(x_k)|, |P'(x_{k+1})| )
\end{equation}
and where $\delta(x)$ was defined in \eqref{delta-def}.  (See \Cref{fig-plower}.)
\end{corollary}

\begin{figure}
    \centering
\includegraphics[width=0.75\textwidth]{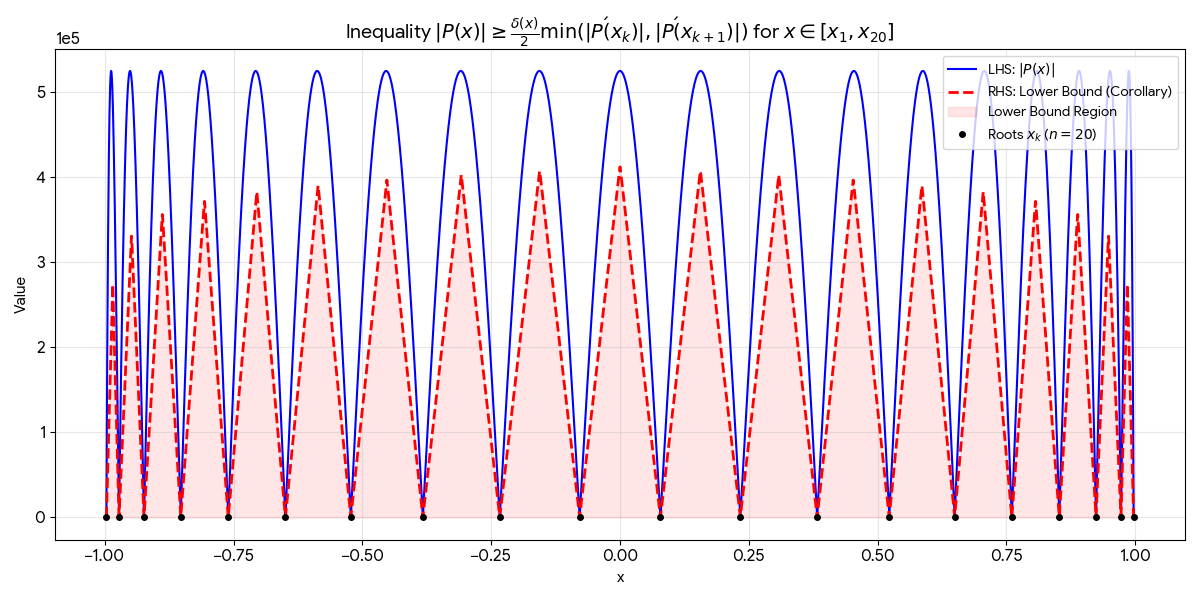}
    \caption{An illustration of \Cref{stosh} with $y=0$, $n=20$, $x_k$ the Chebyshev nodes in \Cref{cheby}, and $k$ chosen so that $x_k \leq x \leq x_{k+1}$ for $x \in [x_1,x_n]$. (Image generated by Gemini.)}
    \label{fig-plower}
\end{figure}

\begin{proof}
First suppose that $y=0$.  From \Cref{consec}, \eqref{lamax}, and the triangle inequality we have that
$$ |P(x)| \left( \frac{1}{|x-x_k| |P'(x_k)|} + \frac{1}{|x-x_{k+1}| |P'(x_{k+1})|}  \right) \geq 1.$$
Since $|x-x_k|, |x-x_{k+1}| \geq \delta(x)$ and $|P'(x_k)|, |P'(x_{k+1})| \geq p_k$, the claim follows.

Now we handle the general case.  Either $x_k$ or $x_{k+1}$ is the closest root to $x+iy$.  If $x_k$ is the closest, we have
$$ \frac{|P(x+iy)|}{\delta(x+iy)} = \prod_{k' \neq k} |x-x_{k'}+iy| \geq \prod_{k' \neq k} |x-x_{k'}| = \frac{|P(x)|}{\delta(x)}.$$
Similarly if $x_{k+1}$ is the closest. The claim follows from the already-established $y=0$ case.
\end{proof}

Now we prove \eqref{jo} for a given $J_0$ as above.  We may of course assume without loss of generality that
$$\int_{J_0} \sum_{k: x_k \in J_0} \frac{|P(x)|}{|P'(x_k)|}\ dx \ll |J_0|^2.$$
Let $x_{k_-} < \dots < x_{k_+}$ be the nodes in $J_0$, then from \eqref{pk} we have
$$\int_{J_0} \sum_{k_- \leq k < k_+} \frac{|P(x)|}{p_k}\ dx \ll |J_0|^2.$$
If we introduce the normalizing constant
\begin{equation}\label{P0-def}
P_0 \coloneqq \frac{|J_0|}{\sum_{k_- \leq k < k_+} \frac{1}{p_k}}
\end{equation}
we thus have
$$ \int_{J_0} \frac{|P(x)|}{P_0}\ dx \ll |J_0|.$$
On subdividing the integral, we conclude that
\begin{equation}\label{subdivide-2}
\sum_{k_- \leq k < k_+} \int_{x_{k}}^{x_{k+1}} \frac{|P(x)|}{P_0}\ dx \ll |J_0|.
\end{equation}
We rewrite this bound further as
\begin{equation}\label{p-2}
    \sum_{k_- \leq k < k_+} b_k a_k \Delta x_k \ll |J_0|
\end{equation}
where
\begin{equation}\label{deltaxk-def}
    \Delta x_k \coloneqq x_{k+1} - x_k
\end{equation}
is the zero spacing,
\begin{equation}\label{pk-def}
a_k \coloneqq \frac{p_k \Delta x_k}{P_0}
\end{equation}
is a normalized version of $p_k \Delta x_k$, and $b_k$ is the ratio
\begin{equation}\label{ak-def}
b_k \coloneqq \frac{\int_{x_k}^{x_{k+1}} |P(x)|\ dx}{p_k (\Delta x_k)^2}.
\end{equation}
Note from \eqref{P0-def}, \eqref{pk-def} that
\begin{equation}\label{p-inv}
    \sum_{k_- \leq k < k_+} \frac{1}{a_k} \Delta x_k = |J_0|
\end{equation}
while from \Cref{stosh} one has
\begin{equation}\label{akb} b_k \gg 1.\end{equation}
From \eqref{akb} and \eqref{subdivide-2} we conclude that
\begin{equation}\label{p-3}
    \sum_{k_- \leq k < k_+} a_k \Delta x_k \ll |J_0|.
\end{equation}
From \Cref{prelim}(ii), (iv) we know that the gap $\Delta_k$ between\footnote{In \cite[p. 192]{esz} a slightly sharper bound of $O(\log n/n)$ is established.} consecutive nodes is bounded by
\begin{equation}\label{deltak-bound} \Delta x_k \ll \frac{\log^4 n}{n}\end{equation}
when at least one of the nodes $x_k$, $x_{k+1}$ lies in $J_0$.  By telescoping series, this implies that
\begin{equation}\label{kdel}
    \sum_{k_- \leq k < k_+} \Delta x_k \asymp |J_0|.
\end{equation}
From \eqref{kdel}, \eqref{p-inv}, and Cauchy--Schwarz we can also obtain the matching lower bound to \eqref{p-3}, thus
\begin{equation}\label{p-3-improv}
    \sum_{k_- \leq k < k_+} a_k \Delta x_k \asymp |J_0|
\end{equation}
and hence by \eqref{akb}, \eqref{p-2}
\begin{equation}\label{p-2-improv}
    \sum_{k_- \leq k < k_+} b_k a_k \Delta x_k \asymp |J_0|
\end{equation}

\begin{remark}
One should think of $P_0$ as the expected size of $P(x)$ on $J_0$, and the above estimates are roughly speaking asserting that $a_k, b_k$ are comparable to $1$ for ``typical'' $k$.  In particular
$$ |P(x)| \asymp a_k P_0 \asymp P_0$$
for ``typical'' $k$ and ``typical'' $x$ in $[x_k, x_{k+1}]$.  The quantity $P_0$ then plays a role somewhat analogous to the amplitude function $A(x)$ in the proof of \Cref{main}(i).
\end{remark}

For $i=1,2,3$, define the intervals
$$ J_i \coloneqq (1-i\eps) J_0$$
so that
$$ J_3 \Subset J_2 \Subset J_1 \Subset J_0.$$
From the gap between nodes in $J_0$, we have
$$ J_1 \subset [x_{k_-}, x_{k_+}] \subset J_0$$
and hence by \eqref{p-2}, \eqref{pk-def}, \eqref{ak-def} we have
\begin{equation}\label{mark} \int_{J_1} |P(x)|\ dx \ll |J_0| P_0 \ll \int_{J_0} |P(x)|\ dx.
\end{equation}

We can relate $P_0$ to the quantity $\alpha$ appearing in \Cref{prelim}:

\begin{lemma}[Approximate value of $P_0$]\label{lem-P0} We have $n^{-1} e^{-n\alpha} \ll P_0 \ll n^{O(1)} e^{-n\alpha}$.  In particular, $\log \frac{1}{P_0} = n\alpha + O(\log n)$.
\end{lemma}

\begin{proof}  From \eqref{muj}, \eqref{goa} we see that $\delta(x) \gg 1/n$ for a set of $x$ in $J_1$ of measure $\gg |J_1| \asymp |J_0|$.  From \eqref{mark} and Markov's inequality (for integrals) one has $|P(x)| \ll P_0$ for at least one such $x$.  Combining this with \eqref{deltax} we obtain the lower bound $P_0 \gg n^{-1} e^{-n\alpha}$.

Next, from \eqref{kdel} and the pigeonhole principle one can find $k_0 \leq k < k_+$ such that $\Delta x_k \gg |J_0|/n$.  From \eqref{p-inv}, we have $\frac{1}{a_k} \ll n$ for that value of $k$, and then by
\Cref{stosh} one has $P(x) \gg \delta(x) P_0 / n$ for $x_k \leq x \leq x_{k+1}$.  Comparing this with \eqref{deltax} we obtain the upper bound $P_0 \ll n^{O(1)} e^{-n\alpha}$.
\end{proof}

We can also control $\log \frac{|P(x)|}{P_0}$ in an $L^2$ sense:

\begin{lemma}[log-$L^2$ norm of $|P(x)|/P_0$]\label{lem-log} We have
    $$ \int_{J_1} \left|\log \frac{|P(x)|}{P_0}\right|^2\ dx \ll |J_0|.$$
\end{lemma}

\begin{proof}  To handle the contribution when $|P(x)| \geq P_0$, we use the pointwise upper bound
$$ \left|\log \frac{|P(x)|}{P_0}\right|^2 \ll \frac{|P(x)|}{P_0}$$
so this portion is acceptable from \eqref{mark}.  For the portion when $|P(x)| < P_0$ and $x_k \leq x < x_{k+1}$, we use \Cref{stosh} to bound
\begin{align*}
 \left|\log \frac{|P(x)|}{P_0}\right|^2 &\leq \log^2 \frac{2 \Delta x_k}{\delta(x) a_k}\\
 &\ll 1 + \log^2 \frac{1}{a_k} + \log^2 \frac{\Delta x_k}{\delta(x)} \\
 &\ll \frac{1}{a_k} + a_k + \log^2 \frac{\Delta x_k}{\delta(x)}.
\end{align*}
The third term has an integral of $O(\Delta x_k)$ on $[x_k,x_{k+1}]$.  Integrating on this interval and then summing in $k$, we see from \eqref{p-inv}, \eqref{p-3}, \eqref{kdel} that this contribution is acceptable.
\end{proof}

\subsection{Selecting a good scale}

We now use some (localized) Littlewood--Paley theory to extend this $L^2$ estimate into the upper half-plane.

\begin{lemma}[Local Littlewood-Paley estimate] We have
$$ \int_0^{n^{\eps^2-1}} \int_{J_3} \left|\frac{P'}{P}(x+iy) + \pi in\rho(x_J)\right|^2\ y dx dy \ll |J_0|$$
where $x_J$ is the center of $J_0$, and $\rho$ is defined by \eqref{rho-def} with $I$ replaced by $I_0$.
\end{lemma}

\begin{proof}  Similarly to \eqref{umo}, we can use the harmonicity of $\log \frac{|P(x+iy)|}{P_0}$ to write
$$ \log \frac{|P(x+iy)|}{P_0} = \int_\R \Poisson_y(x-v) \log \frac{|P(v)|}{P_0}\ dv$$
leading to the splitting
$$ \log \frac{|P(x+iy)|}{P_0} - \pi n \rho(x_J) y = u_1(x+iy) + u_2(x+iy) + u_3(x+iy)$$
where
\begin{align*}
     u_1(x+iy) &\coloneqq \int_{J_1} \Poisson_y(x-v) \log \frac{|P(v)|}{P_0}\ dv \\
     u_2(x+iy) &\coloneqq \int_{I_0 \backslash J_1} \Poisson_y(x-v) \log \frac{|P(v)|}{P_0}\ dv \\
     u_3(x+iy) &\coloneqq \int_{\R \backslash I_0} \Poisson_y(x-v) \log \frac{|P(v)|}{P_0}\ dv - \pi n \rho(x_J) y.
\end{align*}
Applying $\partial_x - i \partial_y$ (and the Cauchy--Riemann equations), we conclude from the triangle inequality that
$$ \left|\frac{P'}{P}(x+iy) + \pi in\rho(x_J)\right|^2 \ll |\nabla u_1(x+iy)|^2 + |\nabla u_2(x+iy)|^2 + |\nabla u_3(x+iy)|^2.$$
From \eqref{littlewood-paley} we have
$$ \int_0^\infty \int_\R |\nabla u_1(x+iy)|^2\ y dx dy = \frac{1}{2} \int_{J_1} \left|\log \frac{|P(x)|}{P_0}\right|^2\ dx$$
so by \Cref{lem-log} it suffices to show that
\begin{equation}\label{u2-est} \int_0^{n^{\eps^2-1}} \int_{J_3} |\nabla u_2(x+iy)|^2\ y dx dy \ll |J_0| \end{equation}
and
\begin{equation}\label{u3-est} \int_0^{n^{\eps^2-1}} \int_{J_3} |\nabla u_3(x+iy)|^2\ y dx dy \ll |J_0|.\end{equation}
As the functions $u_2(x+iy)$ and $u_3(x+iy)$ are harmonic in the rectangle $R( J_2, 2 n^{\eps^2-1} )$ (which contains a $n^{\eps^2-1}$-neighborhood of the region of integration in \eqref{u2-est}, \eqref{u3-est}), it will suffice by elliptic regularity to establish the pointwise bounds
\begin{equation}\label{ellip} u_2(x+iy), u_3(x+iy) \ll 1
\end{equation}
in this rectangle.

In the case of $u_2$, we can use \eqref{deltax} and \Cref{lem-P0} to bound
$$\log \frac{|P(v)|}{P_0} \ll \log n + \log \frac{1}{\delta(v)}$$
for $v \in I_0 \backslash J_1$, and the claim then follows (with room to spare) from the decay of $\Poisson_y$.  For $u_3$, we see from \eqref{rho-def}, \eqref{umu-def}, and \Cref{lem-P0} (and again using the decay of $\Poisson_y$) that
\begin{align*}
 u_3(x+iy) &= \pi n \rho(x+iy) y - \pi n \rho(x_J) y + \int_{\R \backslash I_0} \Poisson_y(x-v) \log \frac{e^{-n\alpha}}{P_0}\ dv\\
&= \pi n (\rho(x+iy) - \rho(x_J)) y + O(1).
\end{align*}
The claim then follows from \eqref{G-lip}.
\end{proof}

By subdividing $y$ into dyadic blocks and using the pigeonhole principle, we can now locate a good scale $\eta$ for the ordinate $y$, for which $P(x+iy)$ behaves like a plane wave:

\begin{corollary}[Good scale for ordinate]\label{dyadic}  There exists $\eta$ in the range
\begin{equation}\label{eta-range}
    n^{\eps^3/2-1} \leq \eta \leq n^{\eps^3-1}
\end{equation}
such that
$$ \int_{\eta/2}^{2\eta} \int_{J_3} \left|\frac{P'}{P}(x+iy) + \pi in\rho(x_J)\right|^2\ dx dy \ll_\eps \frac{|J_0|}{\eta \log n}.$$
\end{corollary}

The point here is the gain of $\frac{1}{\log n}$ over the ``natural'' bound of $O(|J_0|/\eta)$.

Henceforth $\eta$ is chosen as in \Cref{dyadic}.
Observe from the mean-value property of harmonic functions that if
$$ \left|\frac{P'}{P}(x+i\eta) + \pi in\rho(x_J)\right| \geq \frac{1}{\eta \log^{1/3} n}$$
for some $x \in J_4$, then
$$ \int_{\eta/2}^{2\eta} \int_{[x-\eta,x+\eta]} \left|\frac{P'}{P}(x'+iy) + \pi in\rho(x_J)\right|^2\ dx dy \ll \frac{1}{\log^{2/3} n}.$$
We conclude that the bound
\begin{equation}\label{ej}
    \left|\frac{P'}{P}(x+i\eta) + \pi in\rho(x_J)\right| < \frac{1}{\eta \log^{1/3} n}
\end{equation}
holds for all $x \in J_4$ outside of the union of $O_\eps( \frac{|J_0|}{\eta \log^{1/3} n} )$ intervals of length $O(\eta)$.  Informally, this means that $P(x+i\eta)$ locally behaves like a scalar multiple of the plane wave $e^{-\pi i n\rho(x_J) x}$ for $x \in J_4$ outside of a small exceptional set.

\subsection{Localizing to a microscale}

We need to prove \eqref{jo}, which we write as
\begin{equation}\label{jo2}
 \left( \int_{J_0} \frac{|P(x)|}{P_0}\ dx \right) \left( \sum_{k: x_k \in J_0} \frac{P_0}{|P'(x_k)|} \right) \geq \frac{(2-O(\eps))}{\pi^2} |J_4|^2.
\end{equation}
We consider a smooth partition of unity
$$ 1 = \sum_{j \in \Z} \varphi_j(x)$$
where each $\varphi_j$ is smooth, takes values in $[0,1]$, and is supported on an interval $K_j$ of length $\eta \log^{1/6} n$ and obeys the derivative bound
\begin{equation}\label{derivb}
 \varphi'_j(t) \ll (\eta \log^{1/6} n)^{-1}
\end{equation}
for all $t$, with the $K_j$ boundedly overlapping in $j$.  Call an index $j$ \emph{good} if $K_j \Subset J_4$ and \eqref{ej} holds for all $x \in K_j$, and \emph{bad} otherwise.  From the above discussion we see that there are only $O_\eps( \frac{|J_0|}{\eta \log^{1/3} n} )$ bad indices $j$ for which $K_j$ intersects $J_4$.  In particular, if $j$ is restricted to good indices, then $\sum_j \varphi_j(x)$ agrees with $1$ outside of a set of measure $O_\eps( \frac{|J_0|}{\eta \log^{1/6} n} )$, which implies that
\begin{equation}\label{stim}
\int_\R \sum_j \varphi_j(x)\ dx = |J_4| - O_\eps\left( \frac{|J_0|}{\log^{1/6} n} \right) = (1-O(\eps)) |J_4| .
\end{equation}

We will establish the following bounds, analogous to \Cref{trig-lem}:

\begin{lemma}\label{newtrig}\
    \begin{itemize}
\item[(i)] ($L^1$ lower bound) We have
\begin{equation}\label{a-low}
\sum_j \int_\R \frac{|P(x)|}{P_0} \varphi_j(x)\ dx \geq \frac{4 - O(\eps)}{\pi} \int_\R \frac{|P(x+i\eta)|}{P_0 e^{\pi n \eta \rho(x_J)}} \sum_j \varphi_j(x)\ dx + O_\eps\left( \frac{|J_0|}{\log^{1/6} n} \right)
\end{equation}
\item[(ii)] (Lower bound on reciprocal of derivatives) We have
\begin{equation}\label{b-low}
\sum_j \sum_k \frac{P_0}{|P'(x_k)|} \varphi_j(x_k) \geq \frac{1}{2\pi} \int_\R \frac{e^{\pi n \eta \rho(x_J)} P_0}{|P(x+i\eta)|} \sum_j \varphi_j(x)\ dx + O_\eps\left( \frac{|J_0|\log\log n}{\log^{1/6} n} \right)
\end{equation}
\end{itemize}
In both estimates the sum is over good $j$.
\end{lemma}

From \Cref{newtrig}(i), \eqref{mark} we have (assuming $n$ sufficiently large depending on $\eps$)
$$ \int_{J_0} \frac{|P(x)|}{P_0}\ dx \geq \frac{4 - O(\eps)}{\pi} \int_\R \frac{|P(x+i\eta)|}{P_0 e^{\pi n \eta \rho(x_J)}} \sum_j \varphi_j(x)\ dx.$$
Similarly, from \Cref{newtrig}(ii), \eqref{p-inv}
$$ \sum_{k: x_k \in J_0} \frac{P_0}{|P'(x_k)|} \geq \frac{1 - O(\eps)}{2\pi} \int_\R \frac{e^{\pi n \eta \rho(x_J)} P_0}{|P(x+i\eta)|} \sum_j \varphi_j(x)\ dx.$$
By Cauchy--Schwarz, we can then lower bound the left-hand side of \eqref{jo2} by
$$
\frac{2 - O(\eps)}{\pi^2} \left( \int_\R \sum_j \varphi_j(x)\ dx \right)^2.$$
which gives \eqref{jo2} thanks to \eqref{stim}.

It remains to establish \Cref{newtrig}.

\subsection{Lower bounding the \texorpdfstring{$P$}{P} integral}

We now prove \Cref{newtrig}(i), mimicking the proof of \Cref{trig-lem}(i).  We write $P(x+i\eta)$ in polar coordinates as
\begin{equation}\label{polar}
    P(x+i\eta) = |P(x+i\eta)| e^{i\theta(x)}
\end{equation}
for some smooth $\theta \colon J_1 \to \R$ (well-defined up to a multiple of $2\pi$).  From \eqref{ej} and the chain rule we have
\begin{equation}\label{thet} |\theta'(x) + \pi n \rho(x_J)| \ll \frac{1}{\eta \log^{1/3} n}
\end{equation}
for all $x \in 2K_j$.

Let
\begin{equation}\label{M-def}
    M \coloneqq \lfloor 1/\eps \rfloor.
\end{equation}
From \Cref{lem-square} and the triangle inequality we have
\begin{align*}
\int_\R \frac{|P(x)|}{P_0} \varphi_j(x)\ dx &\geq \frac{4}{\pi} \sum_{m \text{ odd}} \frac{(-1)^{(m-1)/2} (1-\frac{m}{M})_+}{m} \int_\R \frac{P(x)}{P_0} \cos(m\theta(x)) \varphi_j(x)\ dx\\
&= \frac{4}{\pi} \Re \sum_{m \text{ odd}} \frac{(-1)^{(m-1)/2} (1-\frac{m}{M})_+}{m} \int_\R \frac{P(x)}{P_0} e^{-im\theta(x)} \varphi_j(x)\ dx.
\end{align*}
Using \Cref{residue-weight}, we have for every $m$ in the above sum that
\begin{align*} \int_\R \frac{P(x)}{P_0} e^{-im\theta(x)} \varphi_j(x)\ dx &= \int_\R \frac{P(x+i\eta)}{P_0} e^{-\pi m n\eta \rho(x_J)} e^{-im\theta(x)} \varphi_j(x)\ dx \\
\quad &+ 2i \int_0^\eta \int_\R \frac{P(x+iy)}{P_0} \partial_{\bar z} (e^{-\pi mn y \rho(x_J)} e^{-i m \theta(x)} \varphi_j(x))\ dx dy.
\end{align*}
In the sum
$$\sum_{m \text{ odd}} \frac{(-1)^{(m-1)/2} (1-\frac{m}{M})_+}{m} \frac{P(x+i\eta)}{P_0} e^{-\pi mn \eta \rho(x_J)} e^{-\pi im\theta(x)},$$
the $m=1$ term can be evaluated using \eqref{polar}, \eqref{M-def} as
$$ \frac{|P(x+i\eta)|}{P_0 e^{\pi n \eta \rho(x_J)}} (1 - O(\eps)),$$
while the contribution of the $m > 1$ terms can be estimated using the triangle inequality as
$$ \ll_\eps \frac{|P(x+i\eta)|}{P_0 e^{3 \pi n \eta \rho(x_J)}}.$$
By \eqref{goa}, \eqref{polar}, this latter term can be absorbed into the error of the $m=1$ term.
To prove \Cref{newtrig}(i), it thus suffices by the triangle inequality to establish the bound
\begin{align*}
    &\sum_j \int_0^\eta \int_\R \frac{|P(x+iy)|}{P_0} |\partial_{\bar z} (e^{-\pi mn y \rho(x_J)} e^{-i m \theta(x)} \varphi_j(x))|\ dx dy \\
    &\quad \ll_\eps |J_0| \log^{-1/6} n
\end{align*}
for each odd $1 \leq m \leq M$.  From the product rule and \eqref{thet}, \eqref{derivb}, the quantity
$$\partial_{\bar z} (e^{-\pi mn y \rho(x_J)} e^{-i m \theta(x)} \varphi_j(x))$$
is supported on the region $x \in K_j$ and bounded by $O_\eps(e^{-\pi n y \rho(x_J)} / \eta \log^{1/3} n)$, so it will suffice to show that
$$ \sum_j \int_0^\eta \int_{K_j} \frac{|P(x+iy)|}{P_0} e^{-\pi n y \rho(x_J)}\ dx dy \ll |J_0| \eta.$$

By the bounded overlap of the $K_j$ and the triangle inequality, it suffices to show that
$$ \int_{J_4} \frac{|P(x+iy)|}{P_0} e^{-\pi n y \rho(x_J)}\ dx \ll |J_0|$$
for each $0 < y \leq \eta$.

To estimate this, we use the pigeonhole principle to find an interval $J_3 \Subset [a_-, a_+] \Subset J_2$ with $\delta(a_-), \delta(a_+) \gg n^{-2}$.  We work with the function $u(x+iy) \coloneqq -U_\mu(x+iy) - \frac{1}{n} \log P_0 - \pi y \rho(x_J)$, which is harmonic in the upper half-plane, and in particular on $R^+([a_-,a_+], n^{\eps^2-1})$.  On the lower edge of this rectangle we have
\begin{equation}\label{udip} u(x) = \frac{1}{n} \log \frac{|P(x)|}{P_0}
\end{equation}For $a_\pm+iy$ in the left or right edges of this rectangle, we have the crude bounds
$$ n^{-O(n)} \ll |P(a_\pm+iy)| \ll e^{O(n)}$$
and hence (by \Cref{lem-P0}, \eqref{alpha})
$$ u(a_\pm+iy) \ll \log n.$$
Finally, for $x + i n^{\eps^2-1}$ on the upper edge of the rectangle, we can use \eqref{umud} and \Cref{lem-P0} to conclude that
$$ u(x + i n^{\eps^2-1}) \ll \frac{\log n}{n}.$$
Applying\footnote{Technically, there is an issue because $u$ has (mild) logarithmic singularities on the lower edge of the rectangle, but this can be dealt with by a standard limiting argument, shifting $u$ by $i\delta$ for a small $\delta>0$ and then taking limits as $\delta \to 0$ using dominated convergence.} \eqref{harmonic}, \Cref{harm-lemma}, and the triangle inequality, we obtain the bound
\begin{equation}\label{ubound}
 u(x+iy) \leq \int_{a_-}^{a^+} \Poisson_y(x-v) \max(u(v),0)\ dv + O\left(\frac{1}{n}\right)
\end{equation}
(with some room to spare in the error term).  Exponentiating this, we conclude that
$$ \frac{|P(x)|}{P_0} e^{-\pi n y \rho(x_J)} \ll \exp\left( \int_{a_-}^{a^+} \Poisson_y(x-v) n\max(u(v),0)\ dv\right)$$
and hence by Jensen's inequality (extending $u(v)$ by zero outside of $[a_-, a_+]$ to make the Poisson kernel have total mass one)
$$ \frac{|P(x)|}{P_0} e^{-\pi n y \rho(x_J)} \ll 1 + \int_{a_-}^{a^+} \Poisson_y(x-v) \exp(n\max(u(v),0))\ dv.$$
Integrating this in $x$, and using the bound $\exp(n\max(u(v),0)) \ll 1 + \frac{|P(v)|}{P_0}$, we conclude that
$$ \int_{J_4} \frac{|P(x+iy)|}{P_0} e^{-\pi n y \rho(x_J)}\ dx \ll |J_0| + \int_{a_-}^{a_+} \frac{|P(v)|}{P_0}\ dv,$$
and the claim follows from \eqref{mark}. This completes the proof of \Cref{newtrig}(i).

\subsection{Lower bounding the \texorpdfstring{$P'$}{P'} sum}

We now prove \Cref{newtrig}(ii); the proof will be somewhat similar to that in the previous section, and also inspired by the proof of \Cref{trig-lem}(ii).  Let $\theta$ and $a_+, a_-$ be as in the previous subsection.  For each good $j$, we apply \Cref{residue-weight} to conclude that
\begin{equation}\label{lime}
\begin{split}
&    \frac{1}{2\pi i} \int_\R \frac{P_0 e^{-\pi n \eta \rho(x_J)} e^{i\theta(x)} \varphi_j(x)}{P(x-i\eta)}\ dx - \frac{1}{2\pi i} \int_\R \frac{P_0 e^{\pi n \eta \rho(x_J)} e^{i\theta(x)} \varphi_j(x)}{P(x+i\eta)}\ dx\\
& \quad = \sum_k \frac{P_0}{P'(x_k)} e^{i\theta(x_k)} \varphi_j(x_k) + \frac{1}{\pi} \int_{-\eta}^{\eta} \int_\R \frac{P_0 \partial_{\bar z} (e^{\pi n y \rho(x_J)} e^{i \theta(x)} \varphi_j(x)}{P(x+iy)}\ dx dy.
\end{split}
\end{equation}
From \eqref{polar} one has
$$ \int_\R \frac{P_0 e^{\pi n \eta \rho(x_J)} e^{i\theta(x)} \varphi_j(x)}{P(x+i\eta)}\ dx = \int_\R \frac{P_0 e^{\pi n \eta \rho(x_J)} \varphi_j(x)}{|P(x+i\eta)|}\ dx.$$
Meanwhile, from \eqref{umud}, \eqref{umu-def}, and \Cref{lem-P0} we have
$$ \frac{1}{n} \log \frac{P_0}{|P(x-i\eta)|} = - \pi \eta \rho(x_J) + O\left( \frac{\log n}{n} \right)$$
and thus
$$ \frac{P_0}{|P(x-i\eta)|} = n^{O(1)} e^{-\pi n \eta \rho(x_J)}$$
which implies (using \eqref{goa} and the triangle inequality) that
$$ \int_\R \frac{P_0 e^{\pi n \eta \rho(x_J)} e^{i\theta(x)} \varphi_j(x)}{P(x-i\eta)}\ dx \ll n^{O(1)} e^{-cn} $$
for some absolute constant $c>0$.  Finally, using \eqref{ej}, \eqref{derivb}, and the Leibniz rule we see that $\partial_{\bar z}( e^{\pi n y \rho(x_J)} e^{i\theta(x)} \varphi_j(x) )$ vanishes unless $x \in K_j$, in which case this expression is $O( \frac{e^{\pi n y \rho(x_J)}}{\eta \log^{1/6} n} )$. Taking imaginary parts in \eqref{lime} and applying the triangle inequality, we conclude that
\begin{align*}
\sum_k \frac{P_0}{|P'(x_k)|} \varphi_j(x_k) &\geq \frac{1}{2\pi} \int_\R \frac{P_0 e^{\pi n \eta \rho(x_J)} \varphi_j(x)}{|P(x+i\eta)|}\ dx\\
&\quad - O\left( \frac{1}{\eta \log^{1/6} n} \int_{-\eta}^{\eta} \int_{K_j} \frac{P_0 e^{\pi n y \rho(x_J)}}{|P(x+iy)|}\ dx dy \right) \\
&\quad - O(n^{O(1)} e^{-cn}  ).
\end{align*}
Summing in $j$, and using the bounded overlap of the $K_j$, we will be done if we can show that
$$\int_{-\eta}^{\eta} \int_{J_4} \frac{P_0 e^{\pi n y \rho(x_J)}}{|P(x+iy)|}\ dx dy \ll \eta |J_0| \log\log n.$$
From \eqref{fsymm} it suffices to show that
\begin{equation}\label{jp}
    \int_{0}^{\eta} \int_{J_4} \frac{P_0 e^{\pi n y \rho(x_J)}}{|P(x+iy)|}\ dx dy \ll \eta |J_0| \log\log n.
\end{equation}
A technical difficulty here is that $\frac{1}{|P(x+iy)|}$ ceases to be locally integrable in $x$ in the limiting case $y=0$, although it remains locally integrable jointly in $x,y$.  To get around this issue we shall treat the contribution of small $y$ separately. Our arguments will be slightly inefficient and lose a factor of $O(\log\log n)$, but we can afford to lose as much as $\log^{1/6} n$ in \eqref{jp} over the expected bound of $O(|J_0|)$, so we will still be able to close the argument.

We first deal with the local contribution when $0 \leq y \leq n^{-1}$.  Here, the $e^{\pi ny \rho(x_J)}$ term is bounded thanks to \eqref{goa}. If we let $x_{k_-} < \dots < x_{k_+}$ be the nodes in $J_3$, then we may bound this contribution by
$$\ll \sum_{k_- \leq k < k_+} \int_0^{n^{-1}} \int_{x_k}^{x_{k+1}} \frac{P_0}{|P(x+iy)|}\ dx dy.$$
By \Cref{stosh} and \eqref{pk-def} one has
\begin{equation}\label{ko} \frac{P_0}{|P(x+iy)|} \ll \frac{\Delta x_k}{a_k} \left( \frac{1}{|x-
x_k+iy|} + \frac{1}{|x-x_{k+1}+iy|} \right)
\end{equation}
and hence by straightforward calculation
$$ \int_0^{n^{-1}} \int_{x_k}^{x_{k+1}} \frac{P_0}{|P(x+iy)|}\ dx dy \ll \frac{\Delta x_k}{na_k}  \max(1, \log(n \Delta x_k)).$$
By \eqref{deltak-bound} one has $1 + \log(n \Delta x_k) \ll \log\log n$.  Using \eqref{p-inv}, we thus see that the size of this contribution is $O( |J_0| \log\log n / n)$, which is acceptable.

Now we deal with the global contribution when $n^{-1} < y \leq \eta$.  It will suffice to show that
$$ \int_{J_4} \frac{P_0}{|P(x+iy)|} e^{\pi n y \rho(x_J)}\ dx \ll |J_0| \log\log n$$
for each such $y$.

We work on the shifted rectangle $R^+([a_-,a_+], n^{\eps^2-1}) \backslash R^+([a_-,a_+], n^{-1})$ (to avoid the poles of $\frac{1}{P}$).
Arguing as in the proof of \eqref{ubound} (but now on the shifted rectangle, and with the function $u$ defined by \eqref{udip} replaced by $-u$), we have
$$ -u(x+iy) \leq \int_{a_-}^{a^+} \Poisson_{y-n^{-1}}(x-v) \max(-u(v+in^{-1}),0)\ dv + O\left(\frac{1}{n}\right)$$
and thus
$$ \frac{P_0}{|P(x+iy)|} e^{\pi n y \rho(x_J)} \ll \exp\left(\int_{a_-}^{a^+} \Poisson_{y-n^{-1}}(x-v) n\max(-u(v+in^{-1}),0)\ dv\right)$$
Using Jensen's inequality and integrating as before, and using $\exp(n\max(-u(v+in^{-1}),0)) \ll 1 + \frac{P_0}{|P(v+in^{-1})|}$, we obtain
$$ \int_{J_4} \frac{P_0}{|P(x+iy)|} e^{\pi n y \rho(x_J)}\ dx \ll |J_0| + \int_{a_-}^{a_+} \frac{P_0}{|P(v+in^{-1})|}\ dv,$$
so it suffices to show that
$$\int_{a_-}^{a_+} \frac{P_0}{|P(x+in^{-1})|}\ dx \ll |J_0| \log\log n.$$
Here we argue as in the local case.  By \eqref{ko} we can bound the left-hand side by
$$ \ll \sum_{k_0 \leq k < k_+} \int_{x_k}^{x_{k+1}} \frac{\Delta x_k}{a_k} \left( \frac{1}{|x-x_k+in^{-1}|} + \frac{1}{|x-x_{k+1}+in^{-1}|} \right)\ dx.$$
Performing the integral and using \eqref{p-inv} and \eqref{deltak-bound}, one can bound this by
$$ \ll \sum_{k_0 \leq k < k_+} \frac{\Delta x_k}{a_k} \log(n \Delta x_k) \ll |J_0| \log\log n.$$
This completes the proof of \Cref{newtrig}(ii), and hence of \Cref{main}(ii).

\bibliographystyle{amsplain}

\end{document}